\documentclass[11pt]{amsart}

\usepackage{amssymb}
\usepackage{overpic}
\usepackage{enumitem}   
\usepackage{graphicx} 
\usepackage{amsrefs}
\usepackage{xcolor}
\usepackage[hidelinks]{hyperref}

\graphicspath{{Figures/}} 

\newtheorem {theorem} {Theorem} 

\newtheorem {corollary} {Corollary}
\newtheorem {lemma} {Lemma}

\newtheorem {remark}{Remark}

\begin{document}
	
\title[Limit sets in control models with large hysteresis]{Limit Sets and Global Bifurcation Structure in Planar Control Models with Large Hysteresis}

\author[Tiago Carvalho,  Leonardo Serantola and Bruno S. Rangel]{Tiago Carvalho$^1$, Leonardo Serantola$^1$ and Bruno de Souza Rangel$^2$}

\address{$^1$IBILCE--UNESP, CEP 15054--000, S. J. Rio Preto, S\~ao Paulo, Brazil}
\address{$^2$UFSCar, CEP 13565--905, S\~ao Carlos, S\~ao Paulo, Brazil}
\email{tiago.carvalho1@unesp.br}
\email{l.serantola@unesp.br}
\email{brunodesouzarangel@gmail.com}



\begin{abstract}

The present paper addresses a problem that may be of considerable interest to a broad audience since the systems considered here operate according to a switching protocol involving two distinct dynamical regimes. Starting from an initial condition, the evolution follows a first vector field until a selected state variable $y$ reaches a lower threshold $C_1$. At this moment, the dynamics switches to a second vector field. The second regime remains active until the same variable attains an upper threshold $C_2>C_1$, when the first vector field is restored. This alternating procedure is then repeated indefinitely giving rise to a piecewise smooth vector field.  A complete characterization of the $\omega$-limit sets is obtained for every admissible combination of parameters and all initial condition. The analysis is carried out by combining explicit solutions of the vector fields with geometric arguments and the  first return map.

Beyond the classification of limit sets, the paper describes the global bifurcation structure of the family. As the parameters vary, the system undergoes qualitative transitions between distinct asymptotic regimes, including the birth  and disappearance of periodic orbits, changes in their stability, the occurrence of continuum of periodic trajectories in degenerate situations, and the replacement of bounded dynamics by monotone zig-zag motions or unbounded trajectories. The corresponding bifurcation diagrams provide a complete qualitative description of the asymptotic dynamics of the model.
    
\end{abstract}

\maketitle

\section{Introduction}

The purpose of this paper is to provide a rigorous mathematical framework for the analysis of limit sets in a class of hysteresis-driven control systems that frequently appears in applications. Although such models are routinely employed in practice, a general theory describing their long-term behavior is, to the best of our knowledge, still lacking. Situations of this type arise naturally in medicine, agriculture, economics, engineering, chemistry, biology and several other scientific disciplines.

The systems considered here operate according to a switching protocol involving two distinct dynamical regimes. Starting from an initial condition, the evolution follows a first vector field until a selected state variable $y$ reaches a lower threshold $C_1$. At this moment, the dynamics switches to a second vector field. The second regime remains active until the same variable attains an upper threshold $C_2>C_1$, when the first vector field is restored. This alternating procedure is then repeated indefinitely giving rise to a piecewise smooth vector field. No general mathematical classification of the limit sets associated with this type of mechanism is currently available. Examples motivated by medical applications can be found in \cites{Nature-Gatenby-2017, eLife-Gatenby-2022, CarCunEuzFlo-RWA-2025}, among others.

A substantial body of literature exists for the particular case $C_1=C_2$, which includes the theory of Filippov systems and related discontinuous dynamical models. In recent years, considerable attention has also been devoted to bifurcation phenomena in this last class of piecewise-smooth systems, especially Hopf-like bifurcations and the associated periodic dynamics; see, for instance, \cites{ShortSimp,LongSimp} and the references therein.
 In particular, bifurcation phenomena and the emergence of periodic dynamics in discontinuous and piecewise-smooth systems, for  $C_1=C_2$, have been extensively investigated; see, for example, \cites{ShortSimp, LongSimp}. Most of the available theory concerns systems whose switching occurs on a single manifold. Comparatively little is known when two distinct switching thresholds are introduced, producing a hysteretic mechanism and a nontrivial switching band. The present paper aims to contribute to this direction by providing a classification of the possible limit sets in a simple, yet representative, hysteretic framework.

 In fact, the situation considered here involves a deliberate hysteresis gap, namely $C_1\neq C_2$, and has received considerably less attention. From an applied perspective, this hysteresis may represent, for instance, a recovery interval between consecutive treatment sessions, allowing the system to evolve for a significant amount of time before a new intervention becomes necessary. Such behavior is fundamentally different from the case $C_1=C_2$, where switching may occur after arbitrarily short time intervals, as happens near invisible fold-fold singularities; see \cite{CarBuzTei-JMPA-2014}.

As a first step toward a general theory, we focus on the simplest nontrivial setting. More precisely, we consider a planar system generated by two linear vector fields and two switching boundaries
\begin{equation}\label{eqL}
    L^{\pm}=\left\{(x,y)\in\mathbb{R}^2: y=\pm\mu\right\},
\end{equation}
where $\mu>0$. The strip bounded by these lines defines the hysteresis region in which the switching mechanism takes place (see Figure \ref{figura_0}). More sophisticated scenarios, including nonlinear vector fields, higher-dimensional phase spaces, multiple modes and curved switching manifolds, will be addressed in future investigations.

Our analysis classify the kind of limit sets for each choice on the parameters in a planar piecewise smooth vector field with large hysteresis. More specifically, trajectories may exhibit:

\begin{enumerate}
\item unbounded motion, with at least one coordinate diverging to $\pm\infty$;

\item a monotone \textit{zig-zag} dynamics inside the hysteresis band, for which the $x$-coordinate tends to $\pm\infty$;

\item a convergence (for positive or negative time) to a  periodic  trajectory intersecting  the hysteresis band. In such a case, either an isolated periodic orbit, i.e., a limit cycle, is obtained or there is a continuum periodic orbits. 
\end{enumerate}

Besides providing a complete classification of the possible $\omega$-limit sets, the present work also describes how these asymptotic behaviors change as the parameters vary. More precisely, the obtained bifurcation diagrams characterize the transitions between the different dynamical regimes. Depending on the parameter configuration, periodic orbits may appear or disappear, their stability may change from attracting to repelling, degenerate situations may give rise to continuum of periodic trajectories, and bounded dynamics may be replaced by monotone zig-zag motions inside the hysteresis band or by unbounded solutions. Consequently, the paper provides not only a classification of limit sets but also a global qualitative description of the bifurcation structure associated with this family of hysteretic control systems.

The remainder of the paper is organized as follows. Section \ref{sec2} introduces the class of hysteretic switching systems studied throughout the work and establishes the necessary notation and preliminary concepts. In Section \ref{sec3} we prove the main results and analyze the corresponding asymptotic regimes, including periodic solutions, monotone \textit{zig-zag} motions and unbounded trajectories. Finally, Section \ref{conc} contains concluding remarks and outlines several directions for future research.

\section{Statement of the Main Results}\label{sec2}

In this section we present the main results concerning the dynamics generated by pairs of vector fields acting  above and below of the large hysteresis band. For $\mu >0$, define the regions
\begin{equation*}
    \Sigma^-= \{ (x,y) \in \mathbb{R}^2 : \,   y \leq - \mu\} \,\, , \,\,  \Sigma^+= \{ (x,y) \in \mathbb{R}^2 : \,   y \geq \mu\} \mbox{ and }
\end{equation*}
\begin{equation}\label{equacao faixa de histerese}
	HB = \{ (x,y) \in \mathbb{R}^2 : \,   - \mu \leq y \leq  \mu\}.
\end{equation}
Throughout the paper we denote by $X^+$ the vector field acting on the region $\Sigma^+$ and by $X^-$ the vector field acting on the region $\Sigma^-$. Both of them are acting in the hysteresis band $HB$ given by \eqref{equacao faixa de histerese}. See Figure~\ref{figura_0}.
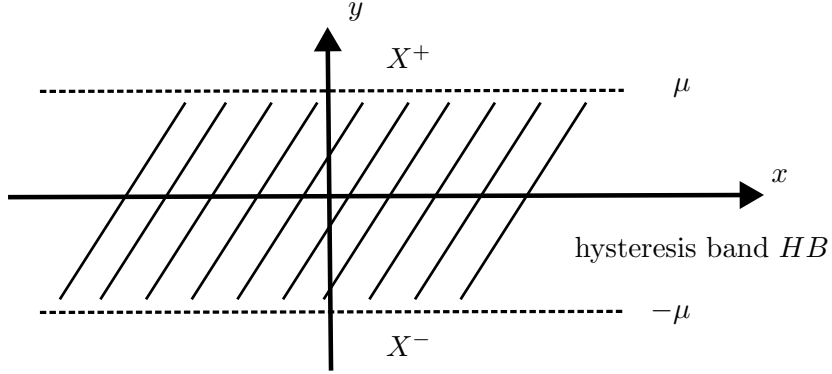
\begin{figure}[ht]
	\begin{center}
		\begin{overpic}[width=10cm]{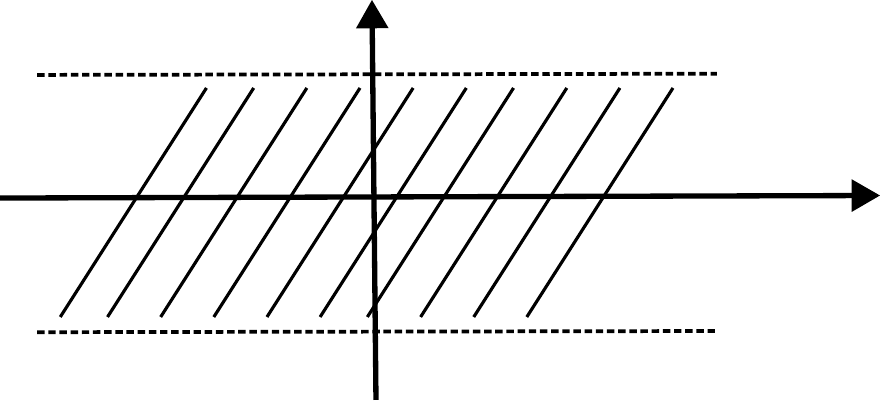}
			\put(88,37){$\mu$}
			\put(85,7){$-\mu$}
			\put(101,25){$x$}
			\put(45,47){$y$}
			\put(50,40){$X^+$}
			\put(50,2){$X^-$}
			\put(75,15){$\mbox{hysteresis band } HB$}
		\end{overpic}
	\end{center}
	\caption{The  configuration studied with an hysteresis band $HB$ between  $y=-\mu$ and $y=\mu$. The vector field $X^+$ acts on the region $\Sigma^+$ and the vector field $X^-$ acts on the region $\Sigma^-$.}\label{figura_0}
\end{figure}

\begin{remark}
	\textbf{IMPORTANT CONVENTION:} In this paper we will consider that when an initial condition of a solution belongs to $HB$, then the vector field to be considered is $X^-$. Physically, we are thinking that in $HB$ the ''treatment'' $X^+$ is not applied. An analogous approach can be done considering the opposite case. 
\end{remark}

\begin{remark}
	We denote by
\[
\overline{\mathbb{R}}
:=
\mathbb{R}\cup\{-\infty,+\infty\}
\]
the extended real line. The extended plane is defined as
\[
\overline{\mathbb{R}}^2
:=
\overline{\mathbb{R}}\times\overline{\mathbb{R}}.
\]
    Consider the vector field $X=(X^+,X^-)$ and an arbitrary initial condition $(x_0,y_0)$. Let us define the $\omega$-limit set of the trajectory through $(x_0,y_0)$ by
\[
\omega\!\left(x_0,y_0\right)
=
\left\{
(\omega_1,\omega_2)\in\overline{\mathbb{R}}^2 :
\lim_{n\to\infty}
X\!\left(t_n;x_0,y_0\right)
=
(\omega_1,\omega_2)
\text{ for  }
t_n\to+\infty
\right\}.
\] 
\end{remark}

Throughout the paper, on the region $\Sigma^+$ we consider the vector field
\begin{equation}\label{EqX+}
	X^+ = (a_1,b_1 x), \quad a_1^2 + b_1^2 \neq 0.
\end{equation}
In this situation, when $a_1b_1\neq 0$,  the vector field $X^+$ presents a fold point over $(0,\mu)$. On the region $\Sigma^-$ we consider a constant vector field
\begin{equation}\label{EqX-}
X^- = (a_2,b_2), \quad a_2,b_2\in\mathbb{R} \quad \textrm{and} \quad a_2^2 + b_2^2 \neq0.
\end{equation}

\begin{theorem}\label{theorem_1} Following the previous convention, consider the concatenated piecewise smooth vector field $Z=(X^+,X^-,HB)$, where  $X^+,X^-,HB$ are given by \eqref{EqX+}, \eqref{EqX-} and \eqref{equacao faixa de histerese}, respectively. 
Given an arbitrary initial condition $(x_0,y_0)$, the  $\omega$-limit set of a trajectory passing through this point is one of the following topological cases:

$\bullet$ $
\{(\pm\infty,\pm\infty)\}$;

$\bullet$ either $
\{(\pm\infty,M)\}$ or $
\{(M,\pm\infty)\}$, with $M$ being a constant. The value of $M$ is determined in the proof of the theorem according to the values of $a_1,a_2,b_1,b_2$;

$\bullet$ limit cycle intersecting $HB$.
 Moreover, the stability of limit cycle is completely determined by the visibility of the fold point of $X^+$, i.e., the sign of $a_1b_1$;

$\bullet$ non-isolated periodic orbit intersecting $HB$. 

$\bullet$  either $
\{(\pm\infty,\mbox{ zig-zag })\}$ or $
\{(\mbox{ zig-zag },\pm\infty)\}$, with the zig-zag occurring inside $HB$.
\end{theorem}

\begin{corollary}
The bifurcation diagrams presented throughout Section \ref{sec3} provide a complete qualitative description of the asymptotic dynamics of the concatenated piecewise smooth vector field $Z=(X^+,X^-,HB)$, where  \linebreak $X^+,X^-,HB$ are given by \eqref{EqX+}, \eqref{EqX-} and \eqref{equacao faixa de histerese}, respectively. 

More precisely, as the parameters vary, the system exhibits the following qualitative transitions:

\begin{enumerate}
\item creation and disappearance of isolated periodic orbits;

\item changes in the stability of periodic orbits from attracting to repelling, and conversely;

\item occurrence of degenerate parameter configurations for which a continuum of periodic trajectories exists;

\item disappearance of periodic dynamics, giving rise to monotone zig-zag trajectories confined to the hysteresis band;

\item transitions between bounded dynamics and trajectories escaping to infinity.
\end{enumerate}

These transitions are summarized by the bifurcation diagrams presented in Figures \ref{figura_01}, \ref{figura_02}, \ref{figura_3.1}, \ref{figura_3.2}, \ref{figura_3.3} and \ref{figura_3.4} which completely organize the parameter regions associated with each qualitative dynamical behavior.
\end{corollary}

\section{Proof of the Main Results}\label{sec3}

\subsection{Proof of Theorem \ref{theorem_1}}

Let us take  the vector field $X^+=\left(a_1, b_1 x\right)$ and $X^-=\left(a_2, b_2\right)$, in regions $\Sigma^+$ and $\Sigma^-$, respectively. See Figure \ref{figura_7} for an example.

Note that the case $b_1=0$ was studied in the previous paper \cite{CarRanSer-Histerese-constante-2026}. Then, let us consider  $b_1\neq0$.

\begin{figure}[ht]
	\begin{center}
		\begin{overpic}[width=6cm]{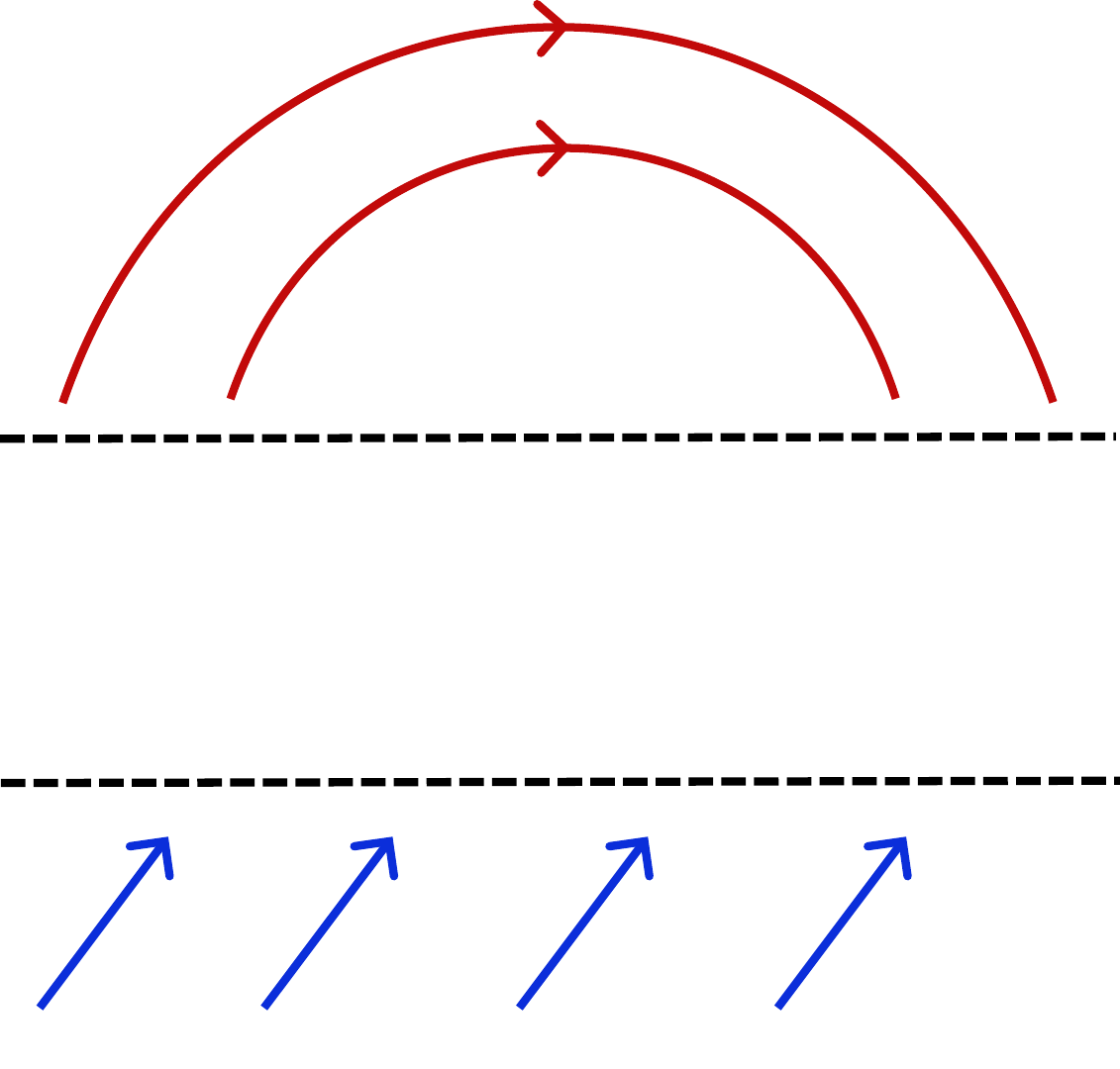}
			\put(106,55){$\mu$}
			\put(103,24){$-\mu$}
			\put(-40,67){$X^+=\left(a_1, b_1 x\right)$}
			\put(-40,10){$X^-=\left(a_2, b_2\right)$}
		\end{overpic}
	\end{center}
\caption{Example of an invisible fold point in the upper region and a constant vector field in the down region of the hysteresis band with $b_1<0$ and $a_1,\, a_2,\, b_2>0$.}\label{figura_7}
\end{figure}
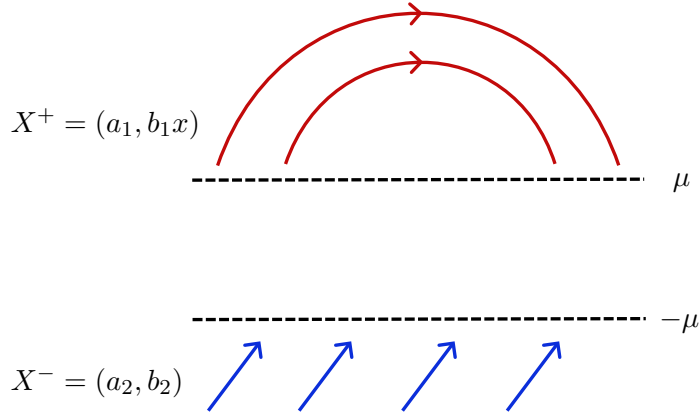

\subsubsection{\textbf{Case 1.1:}} Here we will analyze the dynamics for all cases where $a_1=0$ and $b_1>0$. See Figure \ref{figura_7.1} for the phase portraits corresponding to these cases and Figure \ref{figura_01} for the bifurcation diagram in the variables $a_2$ and $b_2$. 

In the Case 1.1.$i$, $i=1,\cdots,16$, every point of the form $(0,y_0)$, $y_0\geq0$, is an equilibrium point.

\vspace{0.5cm}

$\bullet$ \textbf{Case 1.1.1: $a_1=0, b_1>0, a_2 > 0, b_2 = 0$.} In this case, taking an arbitrary initial condition $(x_0,y_0)$, for $y_0 > \mu$, the trajectory has $\omega$-limit set $\{(+\infty,-\mu)\}$ when $x_0 < 0$ and has $\omega$-limit set $\{(x_0,+\infty)\}$ when $x_0 > 0$. Taking an arbitrary initial condition $(x_0,y_0)$, for $y_0 \leq \mu$, the trajectory has $\omega$-limit set $\{(+\infty,y_0)\}$.

$\bullet$ \textbf{Case 1.1.2: $a_1=0, b_1>0, a_2 > 0, b_2>0$.} In this case, consider the half-line 
\begin{equation}\label{eqR_1}
    R_1=\left\{(x,y)\in\mathbb{R}^2; \,y=\dfrac{b_2}{a_2}x+\mu, \, x<0\right\},
\end{equation}
and the numbers
\begin{equation}\label{eqn_1}
    n_1:=\dfrac{|x_0|}{\left|\dfrac{2a_2\mu}{b_2}\right|} \quad \textrm{and} \quad n_2:=\dfrac{|x_0|}{\left|-\dfrac{2a_2(y_0-2\mu)}{b_2}\right|}.
\end{equation}
Taking an initial condition $(x_0,y_0)$ above the half-line $R_1$, with $y_0\geq\mu$ and $x_0<0$, the trajectory has a monotone zig-zag behavior restricted to HB region, with
\begin{equation}\label{eqN_1}
    N_1:=\left\lceil n_1\right\rceil+1\geq2
\end{equation}
intersections with the line $y=\mu$ and has $\omega$-limit set
$$
\left\{\left(x_0 +\dfrac{2 a_2\mu}{b_2}(N_1-1),+\infty\right)\right\}. 
$$
See Figure \ref{fig-detalhe1} for details. Taking an initial condition $(x_0,y_0)$ above the half-line $R_1$, with $y_0<\mu$ and $x_0<0$, the trajectory has a monotone zig-zag behavior restricted to HB region, with 
\begin{equation}\label{eqN_2}
    N_2:=\left\lceil n_2\right\rceil+1\geq2
\end{equation}
intersections with the line $y=\mu$ and has $\omega$-limit set
$$
\left\{\left(x_0 +\dfrac{2 a_2\mu}{b_2}N_2,+\infty\right)\right\}.
$$
See Figure \ref{fig-detalhe2} for details. Now, taking an initial condition $(x_0,y_0)$ below the half-line $R_1$, the trajectory has $\omega$-limit set
$$
\left\{\left(x_0+\dfrac{a_2 (-y_0 + \mu)}{b_2},+\infty\right)\right\}
$$
when $y_0 \leq \mu$. Taking an initial condition $(x_0,y_0)$, with $y_0>\mu$ and $x_0>0$, the trajectory has $\omega$-limit set $\{(x_0,+\infty)\}$.

\noindent Finally, taking an initial condition $(x_0,y_0)\in R_1$ or a initial condition $(x_0,y_0)$ such that $n_1\in\mathbb{N}$ or $n_2\in\mathbb{N}$, the trajectory has $\omega$-limit set $\{(0,\mu)\}$.

\begin{figure}[ht]
	\begin{center}
		\begin{overpic}[width=5cm]{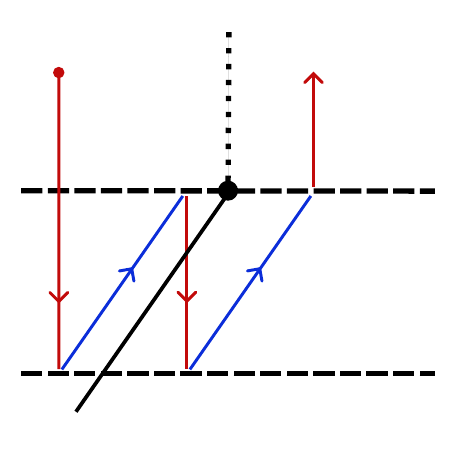}
			\put(0,90){$(x_0,y_0)$}
            \put(6,6){$R_1$}
            \put(100,58){$\mu$}
            \put(98,17){$-\mu$}
		\end{overpic}
	\end{center}
	\caption{An example of the dynamics taking an initial condition $(x_0,y_0)$ above the half-line $R_1$, with $y_0\geq\mu$ and $x_0<0$.}\label{fig-detalhe1}
\end{figure}

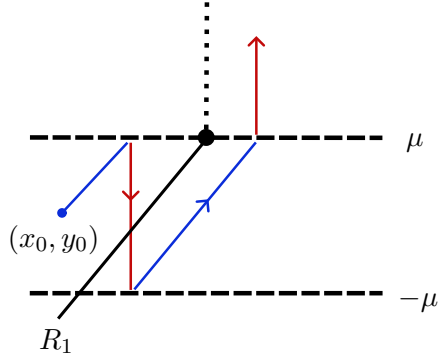
\begin{figure}[ht]
	\begin{center}
		\begin{overpic}[width=5cm]{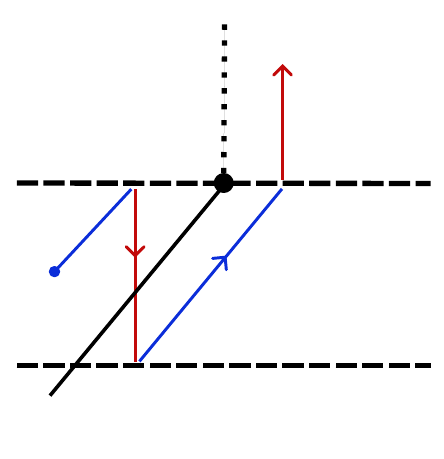}
			\put(-2,32){$(x_0,y_0)$}
            \put(6,6){$R_1$}
            \put(100,58){$\mu$}
            \put(98,17){$-\mu$}
		\end{overpic}
	\end{center}
	\caption{An example of the dynamics taking an initial condition $(x_0,y_0)$ above the half-line $R_1$, with $y_0<\mu$ and $x_0<0$.}\label{fig-detalhe2}
\end{figure}

$\bullet$ \textbf{Case 1.1.3: $a_1=0, b_1>0, a_2 = 0, b_2>0$.} In this case, the HB region is entirely composed of periodic orbits, for $x<0$. So, taking an arbitrary initial condition $(x_0,y_0)$, for $x_0<0$, the trajectory has a periodic orbit as an $\omega$-limit set. Taking an initial condition $(0,y_0)$, the trajectory has $\omega$-limit set $\{(0,\mu)\}$. Now, taking an arbitrary initial condition $(x_0,y_0)$, for $x_0>0$, the trajectory has $\omega$-limit set $\{(x_0,+\infty)\}$.

$\bullet$ \textbf{Case 1.1.4: $a_1=0, b_1>0, a_2 < 0, b_2>0$.} In this case, consider the half-line  
\begin{equation}\label{eqR_2}
    R_2=\left\{(x,y)\in\mathbb{R}^2; \,y=\dfrac{b_2}{a_2}x+\mu, \, x>0\right\}.
\end{equation}
Taking an initial condition $(x_0,y_0)$ with $y_0>\mu$ and $x_0<0$ or an initial condition $(x_0,y_0)$ below the half-line $R_2$, the trajectory has a monotone zig-zag behavior restricted to HB region, with the $x$-coordinate going to $- \infty$. Now, taking an initial condition $(x_0,y_0)$ above the half-line $R_2$, with $x_0>0$, the trajectory has $\omega$-limit set $\{(x_0,+\infty)\}$ when $y_0>\mu$ and the trajectory has $\omega$-limit set
$$
\left\{\left(x_0+\dfrac{a_2 (-y_0 + \mu)}{b_2},+\infty\right)\right\}
$$
when $y_0 \leq \mu$.

\noindent Finally, taking an initial condition $(x_0,y_0)\in R_2$, the trajectory has $\omega$-limit set $\{(0,\mu)\}$.

$\bullet$ \textbf{Case 1.1.5: $a_1=0, b_1>0, a_2<0, b_2=0$.} In this case, taking an arbitrary initial condition $(x_0,y_0)$, for $y_0 > \mu$, the trajectory has $\omega$-limit set $\{(-\infty,-\mu)\}$ when $x_0 < 0$ and has $\omega$-limit set $\{(x_0,+\infty)\}$ when $x_0 > 0$. Taking an arbitrary initial condition $(x_0,y_0)$, for $y_0 \leq \mu$, the trajectory has $\omega$-limit set $\{(-\infty,y_0)\}$.

$\bullet$ \textbf{Case 1.1.6: $a_1=0, b_1>0, a_2<0, b_2<0$.} In this case, taking an arbitrary initial condition $(x_0,y_0)$, for $y_0 > \mu$, the trajectory has $\omega$-limit set $\{(-\infty,-\infty)\}$ when $x_0 < 0$ and has $\omega$-limit set $\{(x_0,+\infty)\}$ when $x_0 > 0$. Taking an arbitrary initial condition $(x_0,y_0)$, for $y_0 \leq \mu$, the trajectory has $\omega$-limit set $\{(-\infty,-\infty)\}$.

$\bullet$ \textbf{Case 1.1.7: $a_1=0, b_1>0, a_2=0, b_2<0$.} In this case, taking an arbitrary initial condition $(x_0,y_0)$, for $y_0 > \mu$, the trajectory has $\omega$-limit set $\{(x_0,-\infty)\}$ when $x_0 < 0$ and has $\omega$-limit set $\{(x_0,+\infty)\}$ when $x_0 > 0$. Taking an arbitrary initial condition $(x_0,y_0)$, for $y_0 \leq \mu$, the trajectory has $\omega$-limit set $\{(x_0,-\infty)\}$.

$\bullet$ \textbf{Case 1.1.8: $a_1=0, b_1>0, a_2>0, b_2<0$.} In this case, taking an arbitrary initial condition $(x_0,y_0)$, for $y_0 > \mu$, the trajectory has $\omega$-limit set $\{(+\infty,-\infty)\}$ when $x_0 < 0$ and has $\omega$-limit set $\{(x_0,+\infty)\}$ when $x_0 > 0$. Taking an arbitrary initial condition $(x_0,y_0)$, for $y_0 \leq \mu$, the trajectory has $\omega$-limit set $\{(+\infty,-\infty)\}$.

\begin{figure}[ht]
	\begin{center}
		\begin{overpic}[width=7cm]{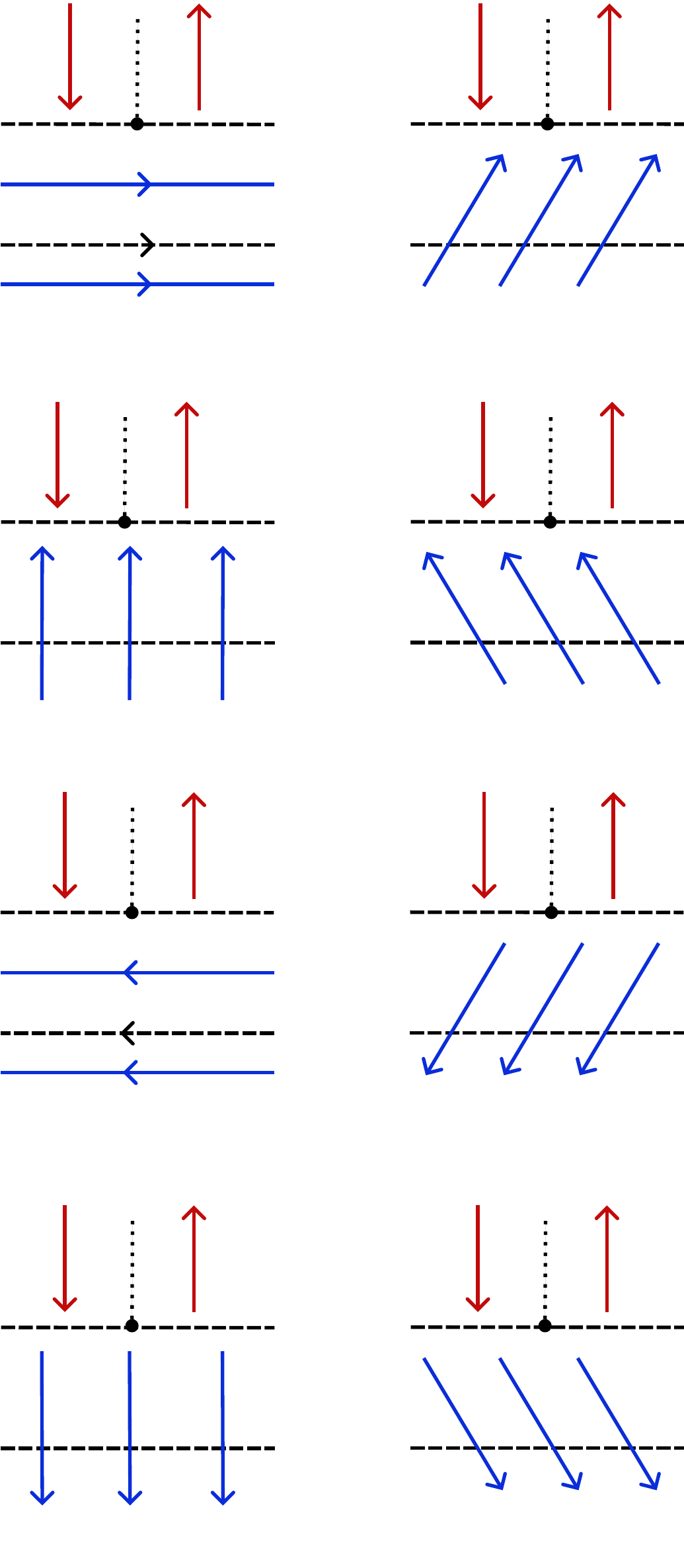}
			\put(4,77){Case $1.1.1$}
            \put(29,77){Case $1.1.2$}
            \put(4,52){Case $1.1.3$}
            \put(29,52){Case $1.1.4$}
            \put(4,26){Case $1.1.5$}
            \put(29,26){Case $1.1.6$}
            \put(4,0){Case $1.1.7$}
            \put(29,0){Case $1.1.8$}
		\end{overpic}
	\end{center}
	\caption{Dynamics for condition $a_1=0$ and $b_1>0$.}\label{figura_7.1}
\end{figure}

\begin{figure}[ht]
	\begin{center}
		\begin{overpic}[width=9cm]{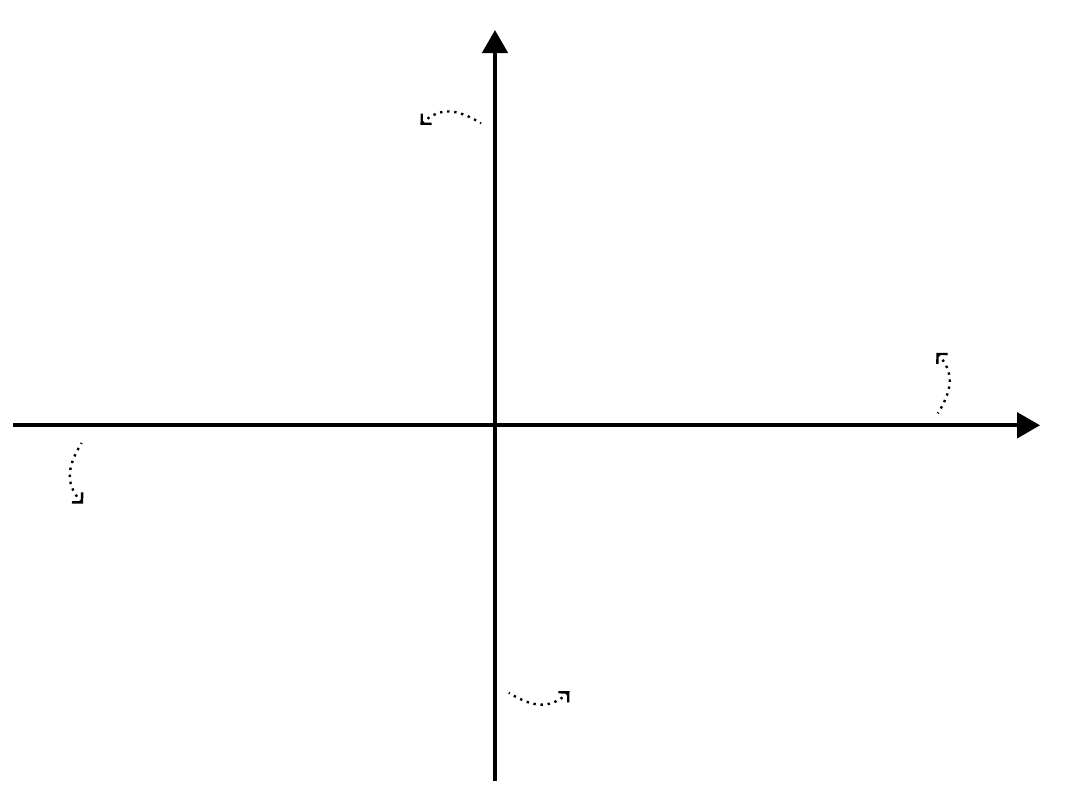}
            \put(97,30){$a_2$}
            \put(48,70){$b_2$}
			\put(75,45){Case $1.1.1$}
            \put(55,55){Case $1.1.2$}
            \put(25,55){Case $1.1.3$}
            \put(5,45){Case $1.1.4$}
            \put(5,20){Case $1.1.5$}
            \put(25,10){Case $1.1.6$}
            \put(55,10){Case $1.1.7$}
            \put(75,20){Case $1.1.8$}
		\end{overpic}
	\end{center}
	\caption{Bifurcation diagram in the variables $a_2$ and $b_2$, for  $a_1=0$, $b_1>0$.}\label{figura_01}
\end{figure}
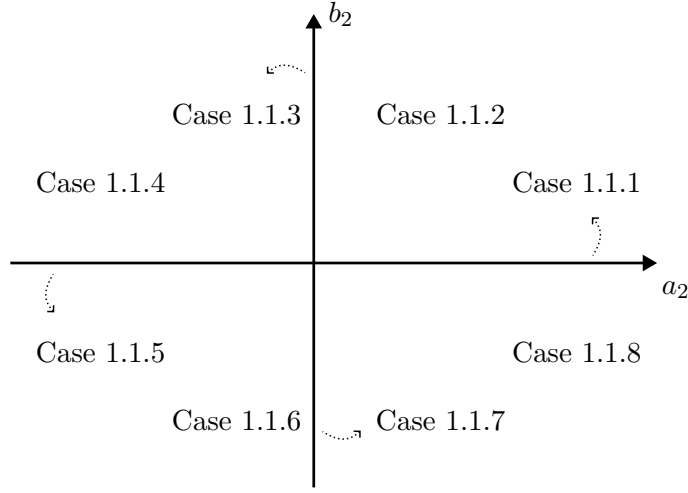

\vspace{0.5cm}

Here we will analyze the dynamics for all cases where $a_1=0$, $b_1<0$. See Figure \ref{figura_7.2} for the phase portraits corresponding to these cases and Figure \ref{figura_02} for the bifurcation diagram in the variables $a_2$ and $b_2$.

\vspace{0.5cm}

$\bullet$ \textbf{Case 1.1.9: $a_1=0, b_1<0, a_2 > 0, b_2 = 0$.} In this case, taking an arbitrary initial condition $(x_0,y_0)$, for $y_0 > \mu$, the trajectory has $\omega$-limit set $\{(x_0,+\infty)\}$ when $x_0 < \mu$ and has $\omega$-limit set $\{(+\infty,-\mu)\}$ when $x_0 > \mu$. Taking an arbitrary initial condition $(x_0,y_0)$, for $y_0 \leq \mu$, the trajectory has $\omega$-limit set $\{(+\infty,y_0)\}$.

$\bullet$ \textbf{Case 1.1.10: $a_1=0, b_1<0, a_2 > 0, b_2>0$.} In this case, consider the half-line $R_1$ given in \eqref{eqR_1}. Taking an initial condition $(x_0,y_0)$ with $y_0>\mu$ and $x_0>0$ or an initial condition $(x_0,y_0)$ below the half-line $R_1$, the trajectory has a monotone zig-zag behavior restricted to HB region, with the $x$-coordinate going to $+ \infty$. Now, taking an initial condition $(x_0,y_0)$ above the half-line $R_1$, with $x_0<0$, the trajectory has $\omega$-limit set $\{(x_0,+\infty)\}$ when $y_0>\mu$ and the trajectory has $\omega$-limit set
$$
\left\{\left(x_0+\dfrac{a_2 (-y_0 + \mu)}{b_2},+\infty\right)\right\}
$$
when $y_0 \leq \mu$.

\noindent Finally, taking an initial condition $(x_0,y_0)\in R_1$, the trajectory has $\omega$-limit set $\{(0,\mu)\}$.

$\bullet$ \textbf{Case 1.1.11: $a_1=0, b_1<0, a_2 = 0, b_2>0$.} In this case,  HB  is entirely composed of periodic orbits, for $x>0$. So, taking an arbitrary initial condition $(x_0,y_0)$, for $x_0>0$, the trajectory has a periodic orbit as an $\omega$-limit set. Taking an initial condition $(0,y_0)$, the trajectory has $\omega$-limit set $\{(0,\mu)\}$ when $y_0\leq \mu$ and $\{(0,y_0)\}$ when $y_0 > \mu$. Now, taking an arbitrary initial condition $(x_0,y_0)$, for $x_0<0$, the trajectory has $\omega$-limit set $\{(x_0,+\infty)\}$.

$\bullet$ \textbf{Case 1.1.12: $a_1=0, b_1<0, a_2 < 0, b_2>0$.} In this case, consider the half-line $R_2$ given in \eqref{eqR_2} and the numbers $n_1$, $n_2$, $N_1$, $N_2$ given in \eqref{eqn_1}, \eqref{eqN_1} and \eqref{eqN_2}, respectively. Taking an initial condition $(x_0,y_0)$ above the half-line $R_2$, with $y_0\geq\mu$ and $x_0>0$, the trajectory has a monotone zig-zag behavior restricted to HB, with $N_1$ intersections with the line $y=\mu$ and has $\omega$-limit set
$$
\left\{\left(x_0 -\dfrac{2 a_2\mu}{b_2}(N_1-1),+\infty\right)\right\}.
$$
See Figure \ref{fig-detalhe3} for details. Taking an initial condition $(x_0,y_0)$ above the half-line $R_2$, with $y_0<\mu$ and $x_0>0$, the trajectory has a monotone zig-zag behavior restricted to HB region, with $N_2$ intersections with the line $y=\mu$ and has $\omega$-limit set
$$
\left\{\left(x_0 -\dfrac{2 a_2\mu}{b_2}N_2,+\infty\right)\right\}.
$$
See Figure \ref{fig-detalhe4} for details. Now, taking an initial condition $(x_0,y_0)$ below the half-line $R_2$, the trajectory has $\omega$-limit set
$$
\left\{\left(x_0+\dfrac{a_2 (-y_0 + \mu)}{b_2},+\infty\right)\right\}
$$
when $y_0 \leq \mu$. Taking an initial condition $(x_0,y_0)$, with $y_0>\mu$ and $x_0<0$, the trajectory has $\omega$-limit set $\{(x_0,+\infty)\}$.

\noindent Finally, taking an initial condition $(x_0,y_0)\in R_2$ or a initial condition $(x_0,y_0)$ such that $n_1\in\mathbb{N}$ or $n_2\in\mathbb{N}$, the trajectory has $\omega$-limit set $\{(0,\mu)\}$.

\begin{figure}[ht]
	\begin{center}
		\begin{overpic}[width=5cm]{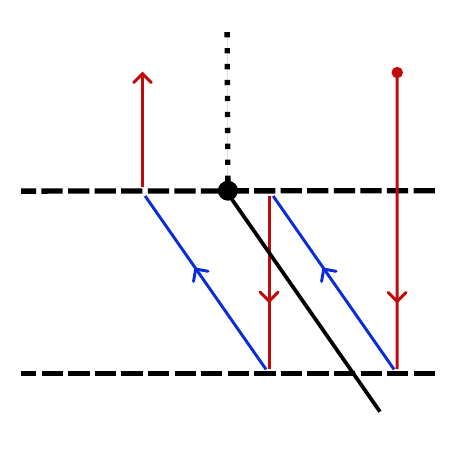}
			\put(75,90){$(x_0,y_0)$}
            \put(85,6){$R_2$}
            \put(-4,58){$\mu$}
            \put(-9,18){$-\mu$}
		\end{overpic}
	\end{center}
	\caption{An example of the dynamics taking an initial condition $(x_0,y_0)$ above the half-line $R_2$, with $y_0\geq\mu$ and $x_0>0$.}\label{fig-detalhe3}
\end{figure}

\begin{figure}[ht]
	\begin{center}
		\begin{overpic}[width=5cm]{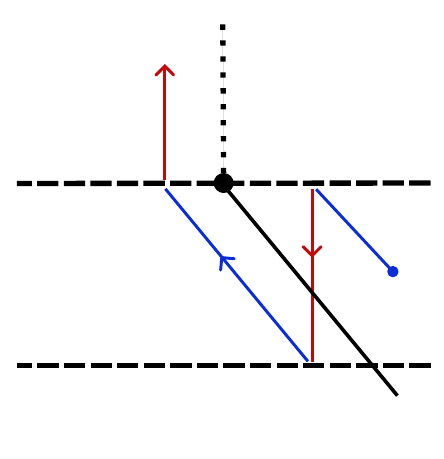}
			\put(75,32){$(x_0,y_0)$}
            \put(85,6){$R_2$}
            \put(-4,58){$\mu$}
            \put(-9,18){$-\mu$}
		\end{overpic}
	\end{center}
	\caption{An example of the dynamics taking an initial condition $(x_0,y_0)$ above the half-line $R_2$, with $y_0<\mu$ and $x_0>0$.}\label{fig-detalhe4}
\end{figure}
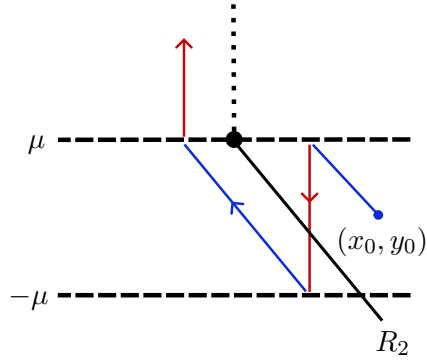

$\bullet$ \textbf{Case 1.1.13: $a_1=0, b_1<0, a_2<0, b_2=0$.} In this case, taking an arbitrary initial condition $(x_0,y_0)$, for $y_0 > \mu$, the trajectory has $\omega$-limit set $\{(x_0,+\infty)\}$ when $x_0 < 0$ and has $\omega$-limit set $\{(-\infty,-\mu)\}$ when $x_0 > 0$. Taking an arbitrary initial condition $(x_0,y_0)$, for $y_0 \leq \mu$, the trajectory has $\omega$-limit set $\{(-\infty,y_0)\}$.

$\bullet$ \textbf{Case 1.1.14: $a_1=0, b_1<0, a_2<0, b_2<0$.} In this case, taking an arbitrary initial condition $(x_0,y_0)$, for $y_0 > \mu$, the trajectory has $\omega$-limit set $\{(x_0,+\infty)\}$ when $x_0 < 0$ and has $\omega$-limit set $\{(-\infty,-\infty)\}$ when $x_0 > 0$. Taking an arbitrary initial condition $(x_0,y_0)$, for $y_0 \leq \mu$, the trajectory has $\omega$-limit set $\{(-\infty,-\infty)\}$.

$\bullet$ \textbf{Case 1.1.15: $a_1=0, b_1<0, a_2=0, b_2<0$.} In this case, taking an arbitrary initial condition $(x_0,y_0)$, for $y_0 > \mu$, the trajectory has $\omega$-limit set $\{(x_0,+\infty)\}$ when $x_0 < 0$ and has $\omega$-limit set $\{(x_0,-\infty)\}$ when $x_0 > 0$. Taking an arbitrary initial condition $(x_0,y_0)$, for $y_0 \leq \mu$, the trajectory has $\omega$-limit set $\{(x_0,-\infty)\}$.

$\bullet$ \textbf{Case 1.1.16: $a_1=0, b_1<0, a_2>0, b_2<0$.} In this case, taking an arbitrary initial condition $(x_0,y_0)$, for $y_0 > \mu$, the trajectory has $\omega$-limit set $\{(x_0,+\infty)\}$ when $x_0 < 0$ and has $\omega$-limit set $\{(+\infty,-\infty)\}$ when $x_0 > 0$. Taking an arbitrary initial condition $(x_0,y_0)$, for $y_0 \leq \mu$, the trajectory has $\omega$-limit set $\{(+\infty,-\infty)\}$.

\begin{figure}[ht]
	\begin{center}
		\begin{overpic}[width=7cm]{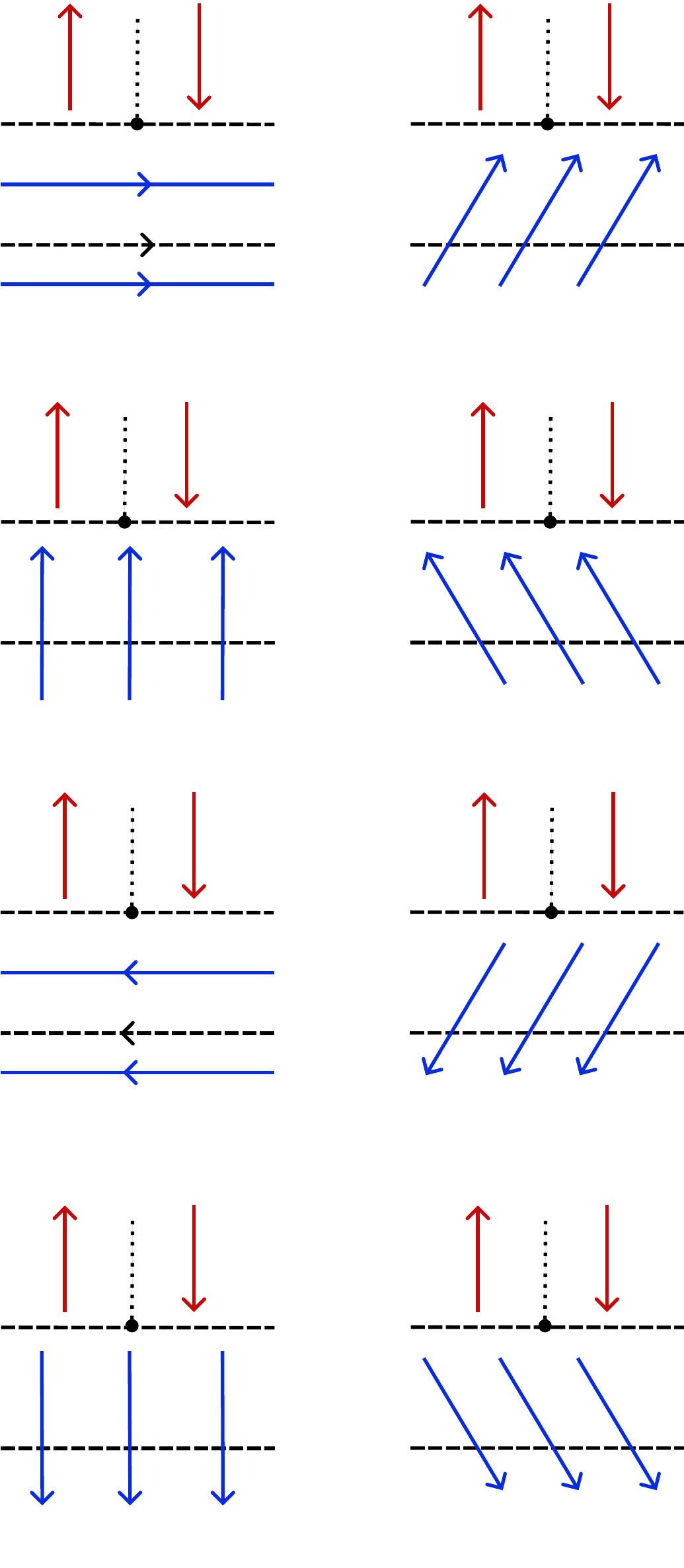}
			\put(4,77){Case $1.1.9$}
            \put(29,77){Case $1.1.10$}
            \put(4,52){Case $1.1.11$}
            \put(29,52){Case $1.1.12$}
            \put(4,26){Case $1.1.13$}
            \put(29,26){Case $1.1.14$}
            \put(4,0){Case $1.1.15$}
            \put(29,0){Case $1.1.16$}
		\end{overpic}
	\end{center}
	\caption{Dynamics for condition $a_1=0$ and $b_1<0$.}\label{figura_7.2}
\end{figure}

\begin{figure}[ht]
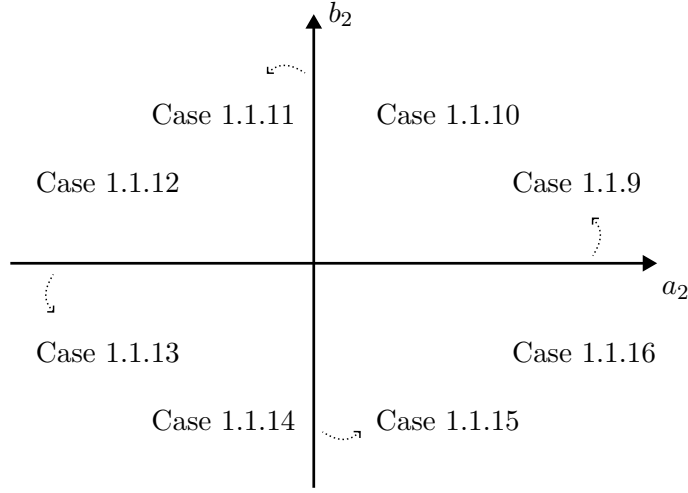

	\begin{center}
		\begin{overpic}[width=9cm]{figura_2.1.pdf}
			\put(97,30){$a_2$}
            \put(48,70){$b_2$}
			\put(75,45){Case $1.1.9$}
            \put(55,55){Case $1.1.10$}
            \put(22,55){Case $1.1.11$}
            \put(5,45){Case $1.1.12$}
            \put(5,20){Case $1.1.13$}
            \put(22,10){Case $1.1.14$}
            \put(55,10){Case $1.1.15$}
            \put(75,20){Case $1.1.16$}
		\end{overpic}
	\end{center}
	\caption{Bifurcation diagram in the variables $a_2$ and $b_2$, for  $a_1=0$, $b_1<0$.}\label{figura_02}
\end{figure}

\vspace{0.5cm}

Finally, let us consider $a_1b_1\neq0$. We will divide the analysis into two subcases:

\textit{\textbf{Case 1.2:}} The vector field $X^+$ has a invisible fold point in $(0,\mu)$, i.e., $a_1>0$, $b_1<0$ or $a_1<0$, $b_1>0$.

\textit{\textbf{Case 1.3:}} The vector field $X^+$ has a visible fold point in $(0,\mu)$, i.e., $a_1>0$, $b_1>0$ or $a_1<0$, $b_1<0$.


\subsubsection{\textbf{Case 1.2:}} Here we will analyze the dynamics for all cases where $a_1>0$ and $b_1<0$. See Figure \ref{figura_8} for the phase portraits corresponding to these cases and Figure \ref{figura_3.1} for the bifurcation diagram in the variables $a_2$ and $b_2$.

\vspace{0.5cm}

$\bullet$ \textbf{Case 1.2.1: $a_1>0, b_1<0, a_2 > 0, b_2 = 0$.} In this case, taking an arbitrary initial condition $(x_0,y_0)$, the trajectory has $\omega$-limit set $\{(+\infty,-\mu)\}$ when $y_0 > \mu$ and has $\omega$-limit set $\{(+\infty,y_0)\}$ when $y_0 \leq \mu$.

$\bullet$ \textbf{Case 1.2.2: $a_1>0, b_1<0, a_2 > 0, b_2>0$.} To analyze the dynamics of the trajectories, we will use Poincaré first return map. See section \ref{subsec2}.

$\bullet$ \textbf{Case 1.2.3: $a_1>0, b_1<0, a_2 = 0, b_2>0$.} To analyze the dynamics of the trajectories, we will use Poincaré first return map. See section \ref{subsec2}.

$\bullet$ \textbf{Case 1.2.4: $a_1>0, b_1<0, a_2 < 0, b_2>0$.} To analyze the dynamics of the trajectories, we will use Poincaré first return map. See section \ref{subsec2}.

$\bullet$ \textbf{Case 1.2.5: $a_1>0, b_1<0, a_2<0, b_2=0$.} In this case, taking an arbitrary initial condition $(x_0,y_0)$, the trajectory has $\omega$-limit set $\{(-\infty,-\mu)\}$ when $y_0 > \mu$ and has $\omega$-limit set $\{(-\infty,y_0)\}$ when $y_0 \leq \mu$.

$\bullet$ \textbf{Case 1.2.6: $a_1>0, b_1<0, a_2<0, b_2<0$.} In this case, taking an arbitrary initial condition $(x_0,y_0)$, the trajectory has $\omega$-limit set $\{(-\infty,-\infty)\}$.

$\bullet$ \textbf{Case 1.2.7: $a_1>0, b_1<0, a_2=0, b_2<0$.} In this case, taking an arbitrary initial condition $(x_0,y_0)$, the trajectory has $\omega$-limit set 
$$
\left\{ \left(\sqrt{x_2+\dfrac{2a_1(\mu-y_0)}{b1}},-\infty\right) \right\}
$$ 
when $y_0 > \mu$ and has $\omega$-limit set $\{(x_0,-\infty)\}$ when $y_0 \leq \mu$.

$\bullet$ \textbf{Case 1.2.8: $a_1>0, b_1<0, a_2>0, b_2<0$.} In this case, taking an arbitrary initial condition $(x_0,y_0)$, the trajectory has $\omega$-limit set $\{(+\infty,-\infty)\}$.

\begin{figure}[ht]
	\begin{center}
		\begin{overpic}[width=7cm]{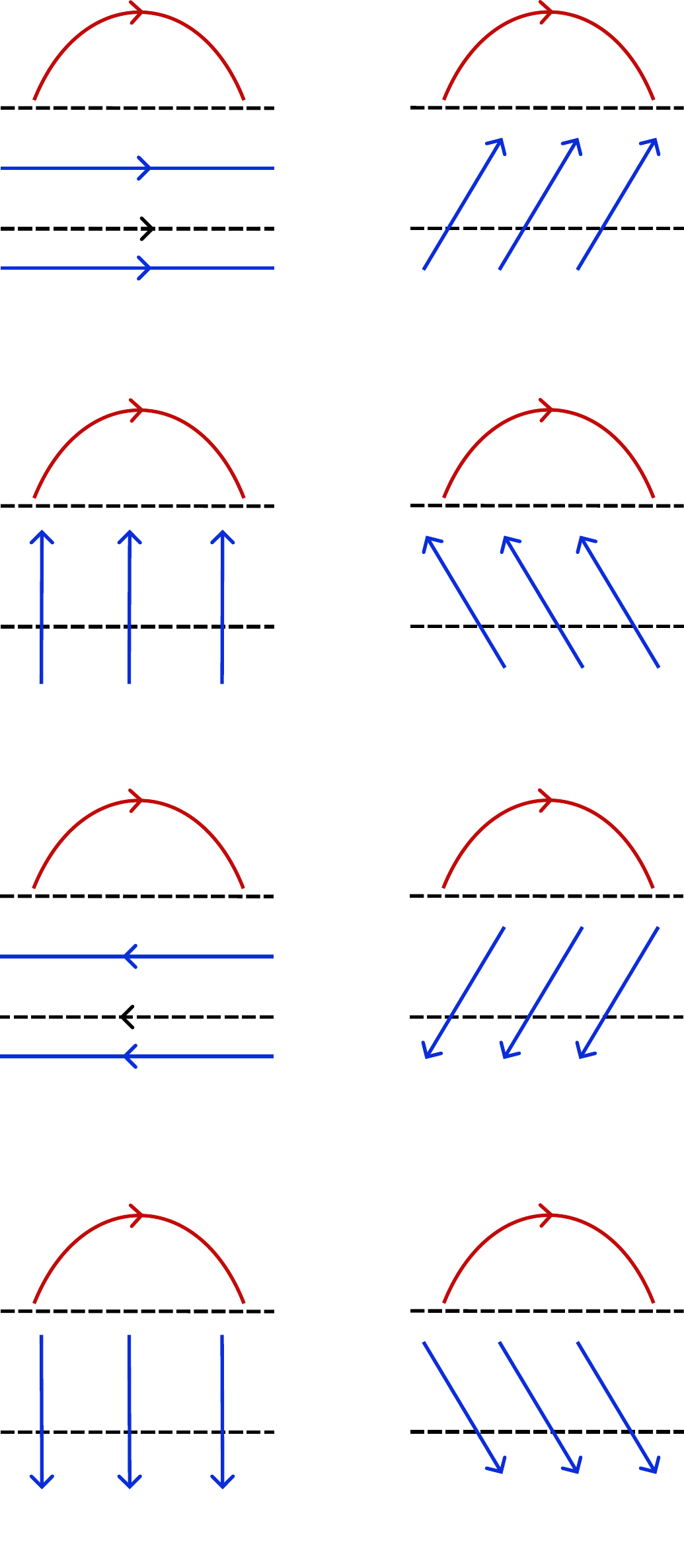}
			\put(4,78){Case $1.2.1$}
            \put(29,78){Case $1.2.2$}
            \put(4,53){Case $1.2.3$}
            \put(29,53){Case $1.2.4$}
            \put(4,28){Case $1.2.5$}
            \put(29,28){Case $1.2.6$}
            \put(4,1){Case $1.2.7$}
            \put(29,1){Case $1.2.8$}
		\end{overpic}
	\end{center}
	\caption{Dynamics for condition $a_1>0$ and $b_1<0$.}\label{figura_8}
\end{figure}
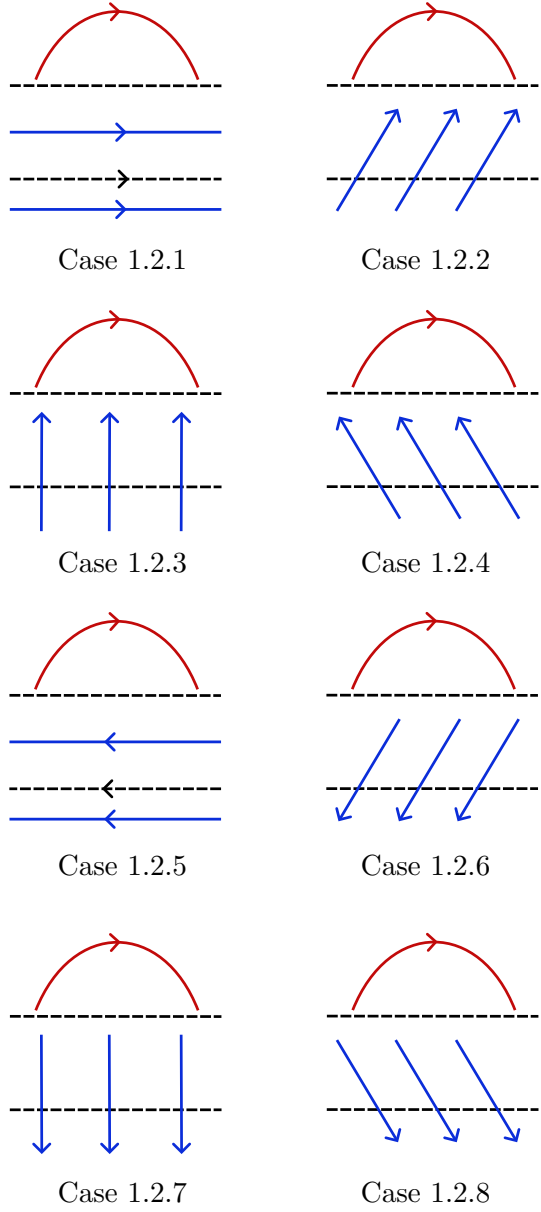

\begin{figure}[ht]
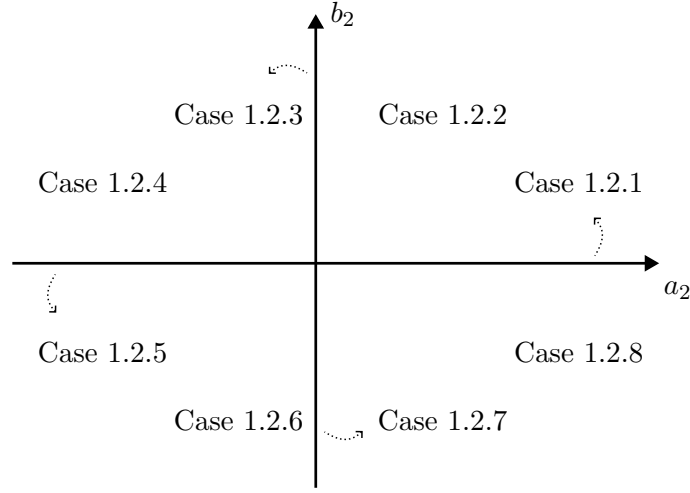

	\begin{center}
		\begin{overpic}[width=9cm]{figura_2.1.pdf}
            \put(97,30){$a_2$}
            \put(48,70){$b_2$}
			\put(75,45){Case $1.2.1$}
            \put(55,55){Case $1.2.2$}
            \put(25,55){Case $1.2.3$}
            \put(5,45){Case $1.2.4$}
            \put(5,20){Case $1.2.5$}
            \put(25,10){Case $1.2.6$}
            \put(55,10){Case $1.2.7$}
            \put(75,20){Case $1.2.8$}
		\end{overpic}
	\end{center}
	\caption{Bifurcation diagram in the variables $a_2$ and $b_2$, for  $a_1>0$, $b_1<0$.}\label{figura_3.1}
\end{figure}

\vspace{0.5cm}

Here we will analyze the dynamics for all cases where $a_1<0$ and $b_1>0$. See Figure \ref{figura_8.1} for the phase portraits corresponding to these cases and Figure \ref{figura_3.2} for the bifurcation diagram in the variables $a_2$ and $b_2$.

\vspace{0.5cm}

$\bullet$ \textbf{Case 1.2.9: $a_1<0, b_1>0, a_2 > 0, b_2 = 0$.} In this case, taking an arbitrary initial condition $(x_0,y_0)$, the trajectory has $\omega$-limit set $\{(+\infty,-\mu)\}$ when $y_0 > \mu$ and has $\omega$-limit set $\{(+\infty,y_0)\}$ when $y_0 \leq \mu$.

$\bullet$ \textbf{Case 1.2.10: $a_1<0, b_1>0, a_2 > 0, b_2>0$.} To analyze the dynamics of the trajectories, we will use Poincaré first return map. See section \ref{subsec2}.

$\bullet$ \textbf{Case 1.2.11: $a_1<0, b_1>0, a_2 = 0, b_2>0$.} To analyze the dynamics of the trajectories, we will use Poincaré first return map. See section \ref{subsec2}.

$\bullet$ \textbf{Case 1.2.12: $a_1<0, b_1>0, a_2 < 0, b_2>0$.} To analyze the dynamics of the trajectories, we will use Poincaré first return map. See section \ref{subsec2}.

$\bullet$ \textbf{Case 1.2.13: $a_1<0, b_1>0, a_2<0, b_2=0$.} In this case, taking an arbitrary initial condition $(x_0,y_0)$, the trajectory has $\omega$-limit set $\{(-\infty,-\mu)\}$ when $y_0 > \mu$ and has $\omega$-limit set $\{(-\infty,y_0)\}$ when $y_0 \leq \mu$.

$\bullet$ \textbf{Case 1.2.14: $a_1<0, b_1>0, a_2<0, b_2<0$.} In this case, taking an arbitrary initial condition $(x_0,y_0)$, the trajectory has $\omega$-limit set $\{(-\infty,-\infty)\}$.

$\bullet$ \textbf{Case 1.2.15: $a_1<0, b_1>0, a_2=0, b_2<0$.} In this case, taking an arbitrary initial condition $(x_0,y_0)$, the trajectory has $\omega$-limit set 
$$
\left\{ \left(\sqrt{x_2+\dfrac{2a_1(\mu-y_0)}{b1}},-\infty\right)\right\}
$$
when $y_0 > \mu$ and has $\omega$-limit set $\{(x_0,-\infty)\}$ when $y_0 \leq \mu$.

$\bullet$ \textbf{Case 1.2.16: $a_1<0, b_1>0, a_2>0, b_2<0$.} In this case, taking an arbitrary initial condition $(x_0,y_0)$, the trajectory has $\omega$-limit set $\{(+\infty,-\infty)\}$.

\begin{figure}[ht]
	\begin{center}
		\begin{overpic}[width=7cm]{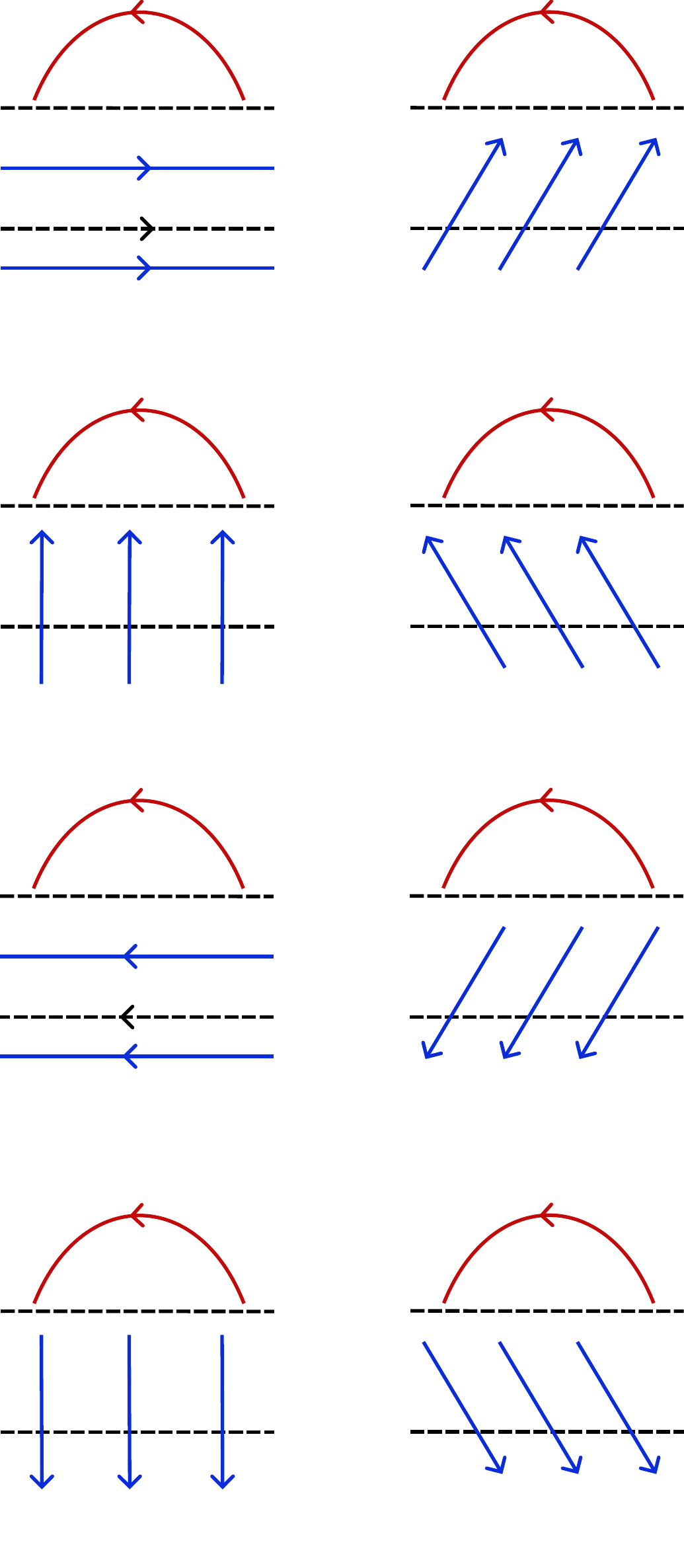}
			\put(4,78){Case $1.2.9$}
            \put(29,78){Case $1.2.10$}
            \put(4,53){Case $1.2.11$}
            \put(29,53){Case $1.2.12$}
            \put(4,28){Case $1.2.13$}
            \put(29,28){Case $1.2.14$}
            \put(4,1){Case $1.2.15$}
            \put(29,1){Case $1.2.16$}
		\end{overpic}
	\end{center}
	\caption{Dynamics for condition $a_1<0$ and $b_1>0$.}\label{figura_8.1}
\end{figure}

\begin{figure}[ht]
	\begin{center}
		\begin{overpic}[width=9cm]{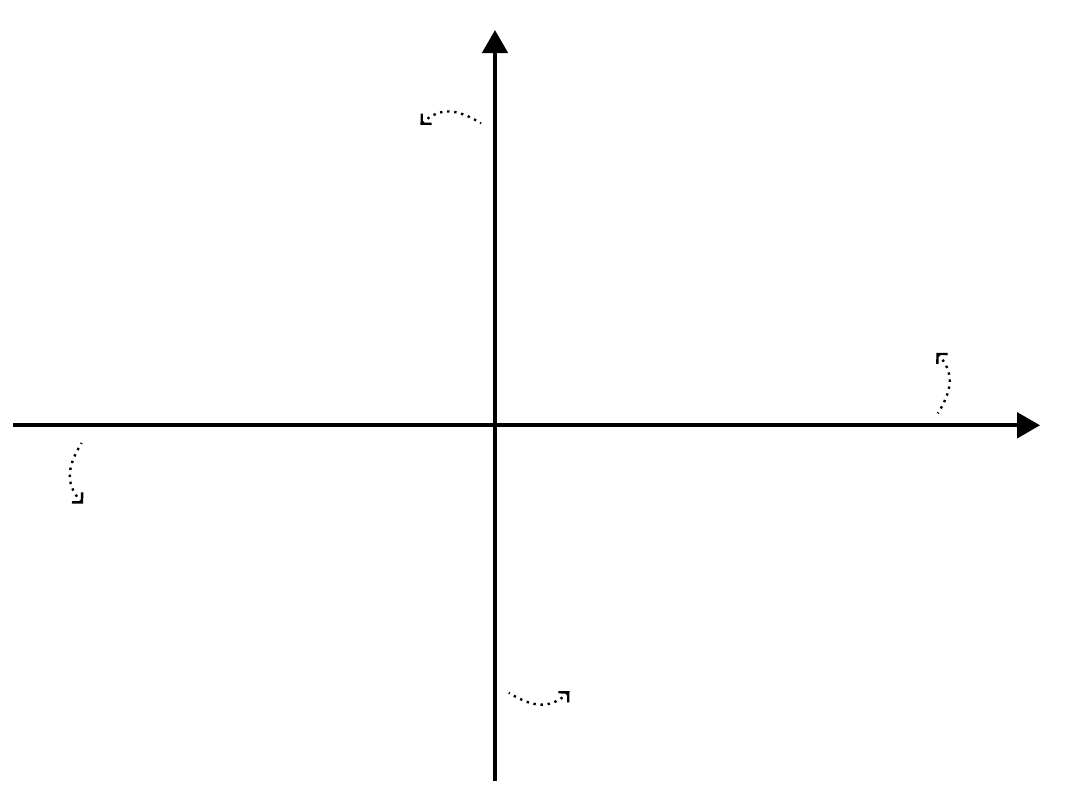}
			\put(97,30){$a_2$}
            \put(48,70){$b_2$}
			\put(75,45){Case $1.2.9$}
            \put(55,55){Case $1.2.10$}
            \put(22,55){Case $1.2.11$}
            \put(5,45){Case $1.2.12$}
            \put(5,20){Case $1.2.13$}
            \put(22,10){Case $1.2.14$}
            \put(55,10){Case $1.2.15$}
            \put(75,20){Case $1.2.16$}
		\end{overpic}
	\end{center}
	\caption{Bifurcation diagram in the variables $a_2$ and $b_2$, for  $a_1<0$, $b_1>0$.}\label{figura_3.2}
\end{figure}

\subsubsection{\textbf{Case 1.3:}} We will analyze the dynamics for all cases where $a_1>0$ and $b_1>0$. See Figure \ref{figura_10} for the phase portraits corresponding to these cases and Figure \ref{figura_3.3} for the bifurcation diagram in the variables $a_2$ and $b_2$. Define
$$
p_1^-=\left(-\dfrac{2a_1\sqrt{\mu}}{\sqrt{a_1b_1}},\mu\right) \quad \textrm{and} \quad p_1^+=\left(\dfrac{2a_1\sqrt{\mu}}{\sqrt{a_1b_1}},\mu\right).
$$

The orbit of \(X^+\) through $p_0=(0,-\mu)$ intersects the section
$L^+$ at two points, namely \(p_1^-\) and \(p_1^+\).

We denote by \(R_1^-\) (respectively \(R_1^+\))
the connected component of the orbit of \(X^+\) joining
$p_0$ to \(p_1^-\) (respectively \(p_1^+\)). See Figure \ref{figura_11}.

\begin{figure}[ht]
	\begin{center}
		\begin{overpic}[width=8cm]{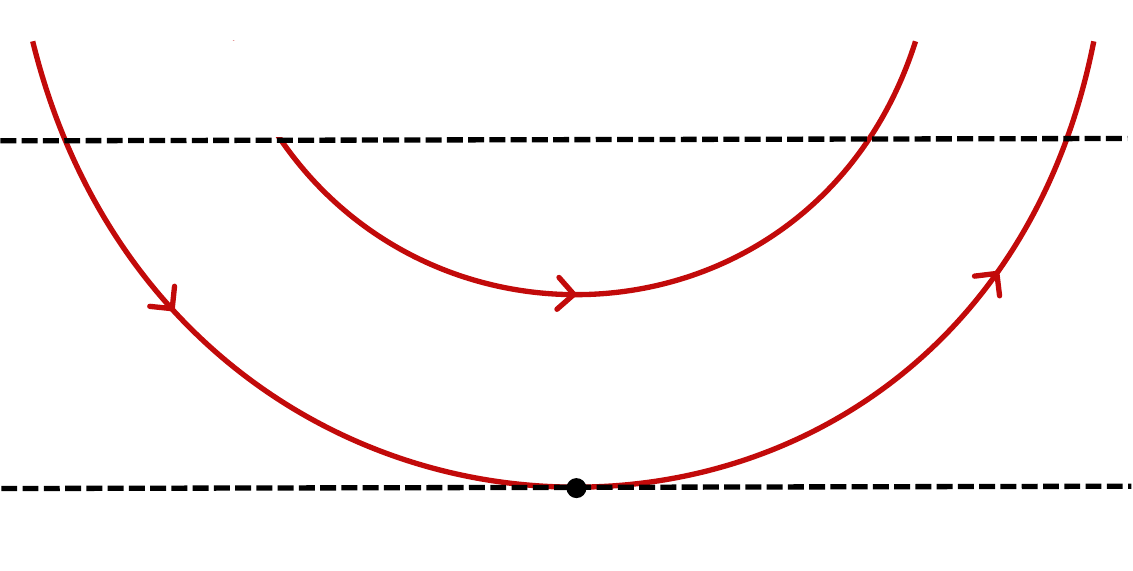}
            \put(5,45){$R_1^-$}
            \put(88,45){$R_1^+$}
            \put(45,3){$p_0$}
            \put(0,32){$p_1^-$}
            \put(95,32){$p_1^+$}
		   \end{overpic}
	\end{center}
\caption{Dynamics between the points $p_1^-$ and $p_1^+$.}\label{figura_11}
\end{figure}
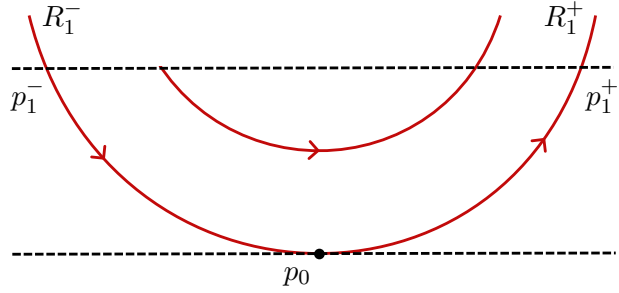

The curve \(R_1^-\) separates \(\Sigma^+\) into two connected components. By the \emph{left-hand side of \(R_1^-\)} we mean the left connected component together with the curve \(R_1^-\) itself. Analogously, the \emph{right-hand side of \(R_1^-\)} denotes the right connected component together with the curve \(R_1^-\).

\vspace{0.5cm}

$\bullet$ \textbf{Case 1.3.1: $a_1>0, b_1>0, a_2 > 0, b_2 = 0$.} In this case, taking an initial condition $(x_0,y_0)$ with $y_0>\mu$ on the left-hand side of \(R_1^-\), the trajectory has $\omega$-limit set $\{(+\infty,-\mu)\}$. Taking an initial condition $(x_0,y_0)$ with $y_0>\mu$ on the right-hand side of \(R_1^-\), the trajectory has $\omega$-limit set $\{(+\infty,+\infty)\}$. Finally, taking an initial condition $(x_0,y_0)$ with $y_0\leq\mu$, the trajectory has $\omega$-limit set $\{(+\infty,y_0)\}$.

$\bullet$ \textbf{Case 1.3.2: $a_1>0, b_1>0, a_2 > 0, b_2>0$.} To analyze the dynamics of the trajectories, we will use the Poincaré first return map. See section \ref{subsec2}.

$\bullet$ \textbf{Case 1.3.3: $a_1>0, b_1>0, a_2 = 0, b_2>0$.} To analyze the dynamics of the trajectories, we will use the Poincaré first return map. See section \ref{subsec2}.

$\bullet$ \textbf{Case 1.3.4: $a_1>0, b_1>0, a_2 < 0, b_2>0$.} To analyze the dynamics of the trajectories, we will use the Poincaré first return map. See section \ref{subsec2}.

$\bullet$ \textbf{Case 1.3.5: $a_1>0, b_1>0, a_2<0, b_2=0$.} In this case, taking an initial condition $(x_0,y_0)$ with $y_0>\mu$ on the left-hand side of \(R_1^-\), the trajectory has $\omega$-limit set $\{(-\infty,-\mu)\}$. Taking an initial condition $(x_0,y_0)$ with $y_0>\mu$ on the right-hand side of \(R_1^-\), the trajectory has $\omega$-limit set $\{(+\infty,+\infty)\}$. Finally, taking an initial condition $(x_0,y_0)$ with $y_0\leq\mu$, the trajectory has $\omega$-limit set $\{(-\infty,y_0)\}$.

$\bullet$ \textbf{Case 1.3.6: $a_1>0, b_1>0, a_2<0, b_2<0$.} In this case, taking an initial condition $(x_0,y_0)$ with $y_0>\mu$ on the left-hand side of \(R_1^-\), the trajectory has $\omega$-limit set $\{(-\infty,-\infty)\}$. Taking an initial condition $(x_0,y_0)$ with $y_0>\mu$ on the right-hand side of \(R_1^-\), the trajectory has $\omega$-limit set $\{(+\infty,+\infty)\}$. Finally, taking an initial condition $(x_0,y_0)$ with $y_0\leq\mu$, the trajectory has $\omega$-limit set $\{(-\infty,-\infty)\}$.

$\bullet$ \textbf{Case 1.3.7: $a_1>0, b_1>0, a_2=0, b_2<0$.} In this case, taking an initial condition $(x_0,y_0)$ with $y_0>\mu$ on the left-hand side of \(R_1^-\), the trajectory has $\omega$-limit set
\[
\left\{ \left(\sqrt{x_2+\dfrac{2a_1(\mu-y_0)}{b1}},-\infty\right) \right\}.
\]
Taking an initial condition $(x_0,y_0)$ with $y_0>\mu$ on the right-hand side of \(R_1^-\), the trajectory has $\omega$-limit set $\{(+\infty,+\infty)\}$. Finally, taking an initial condition $(x_0,y_0)$, with $y_0\leq\mu$, the trajectory has $\omega$-limit set $\{(x_0,-\infty)\}$.

$\bullet$ \textbf{Case 1.3.8: $a_1>0, b_1>0, a_2>0, b_2<0$.} In this case, taking an initial condition $(x_0,y_0)$ with $y_0>\mu$ on the left-hand side of \(R_1^-\), the trajectory has $\omega$-limit set $\{(+\infty,-\infty)\}$. Taking an initial condition $(x_0,y_0)$ with $y_0>\mu$ on the right-hand side of \(R_1^-\), the trajectory has $\omega$-limit set $\{(+\infty,+\infty)\}$. Finally, taking an initial condition $(x_0,y_0)$ with $y_0\leq\mu$, the trajectory has $\omega$-limit set $\{(+\infty,-\infty)\}$.

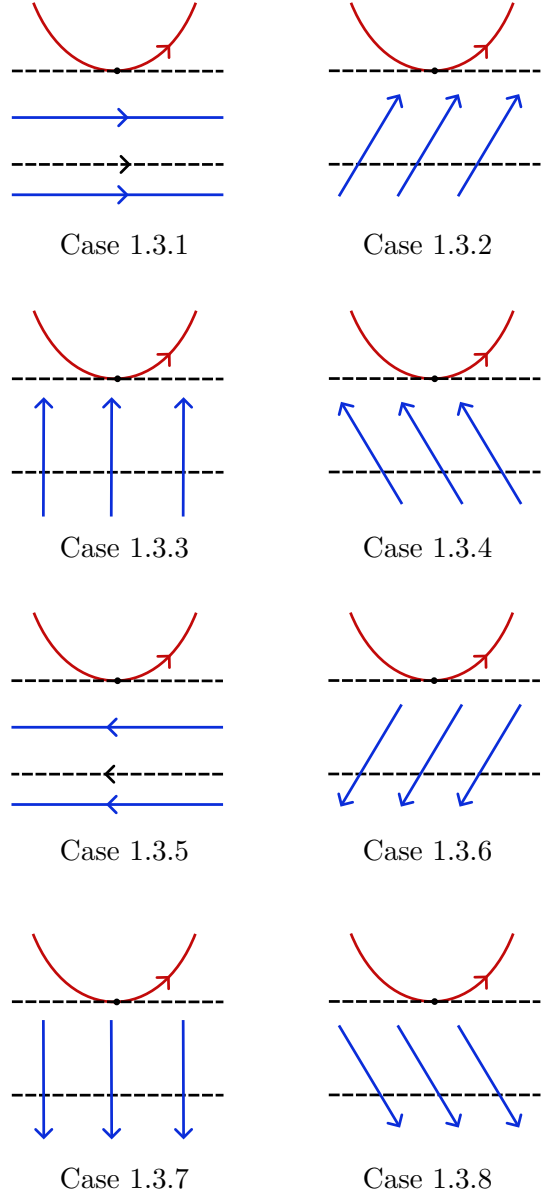
\begin{figure}[ht]
	\begin{center}
		\begin{overpic}[width=7cm]{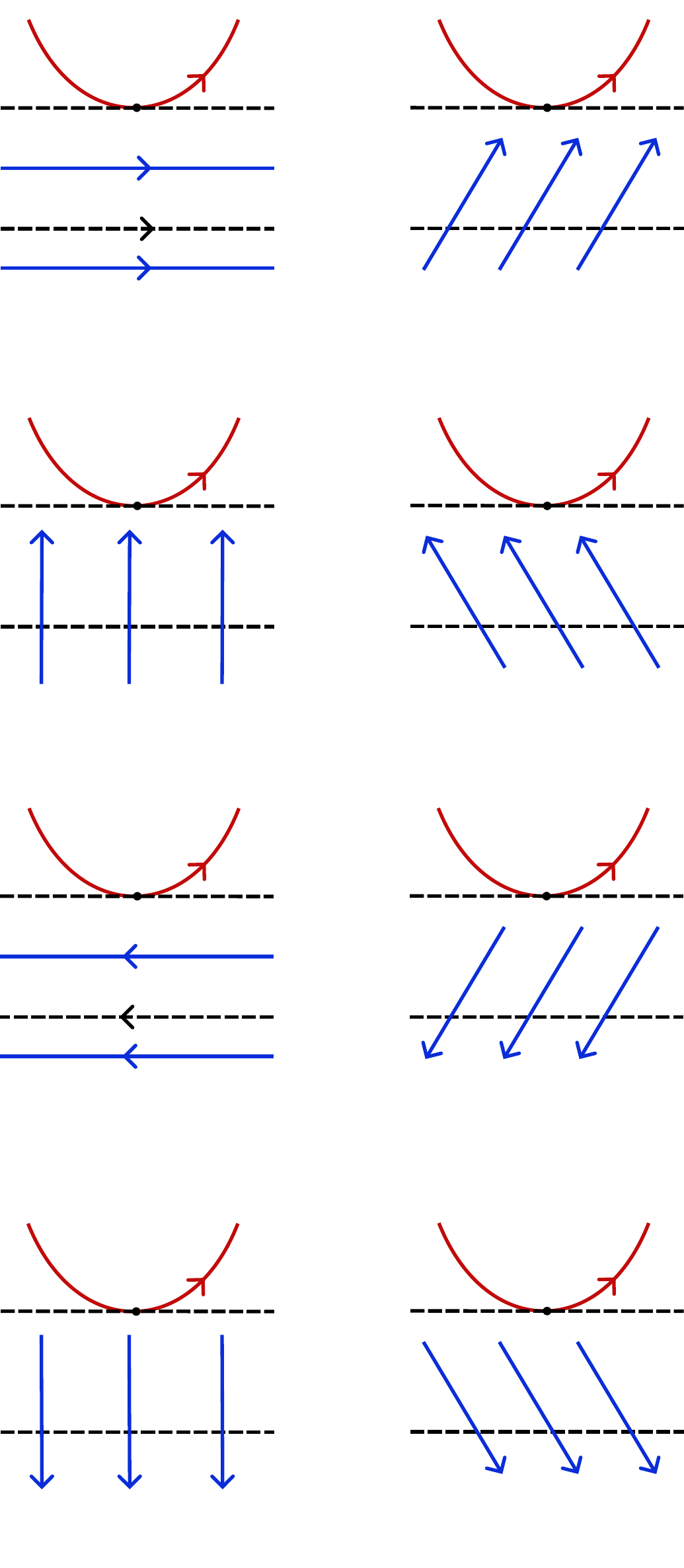}
			\put(4,78){Case $1.3.1$}
            \put(29,78){Case $1.3.2$}
            \put(4,53){Case $1.3.3$}
            \put(29,53){Case $1.3.4$}
            \put(4,28){Case $1.3.5$}
            \put(29,28){Case $1.3.6$}
            \put(4,1){Case $1.3.7$}
            \put(29,1){Case $1.3.8$}
		\end{overpic}
	\end{center}
	\caption{Dynamics for condition $a_1>0$ and $b_1>0$.}\label{figura_10}
\end{figure}

\begin{figure}[ht]
	\begin{center}
		\begin{overpic}[width=9cm]{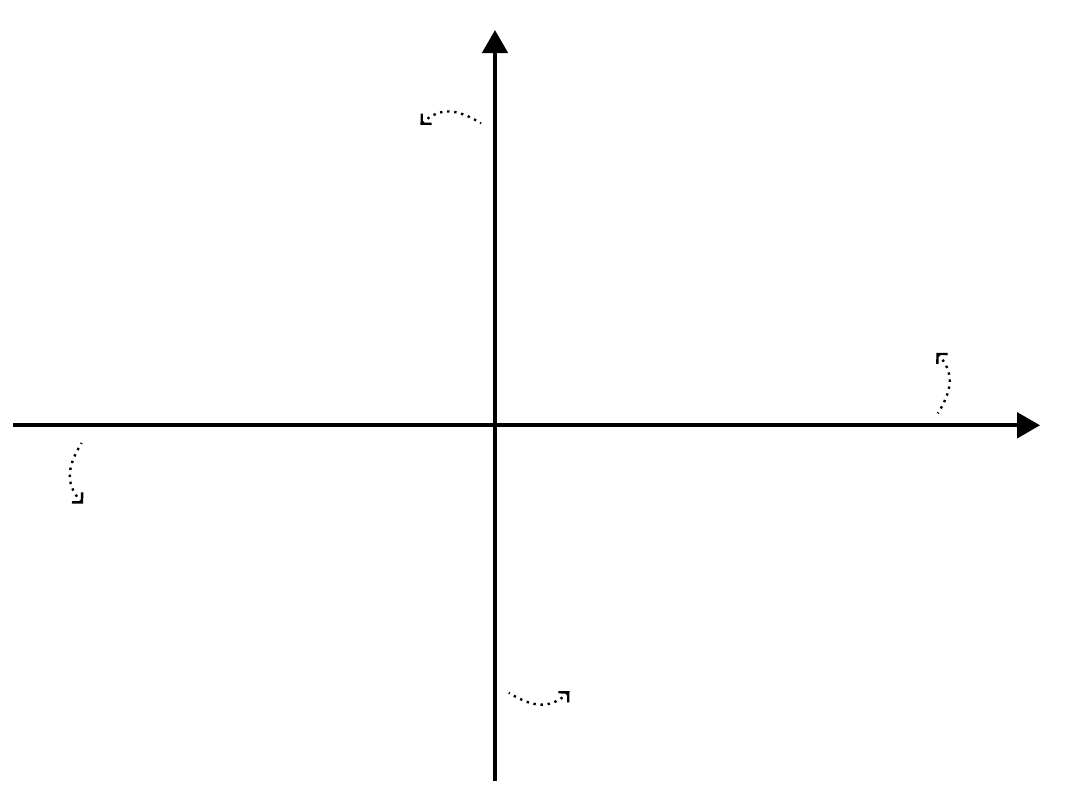}
            \put(97,30){$a_2$}
            \put(48,70){$b_2$}
			\put(75,45){Case $1.3.1$}
            \put(55,55){Case $1.3.2$}
            \put(25,55){Case $1.3.3$}
            \put(5,45){Case $1.3.4$}
            \put(5,20){Case $1.3.5$}
            \put(25,10){Case $1.3.6$}
            \put(55,10){Case $1.3.7$}
            \put(75,20){Case $1.3.8$}
		\end{overpic}
	\end{center}
	\caption{Bifurcation diagram in the variables $a_2$ and $b_2$, for  $a_1>0$, $b_1>0$.}\label{figura_3.3}
\end{figure}

\vspace{0.5cm}

Now, we will analyze the dynamics for all cases where $a_1<0$ and $b_1<0$. See Figure \ref{figura_10.1} for the phase portraits corresponding to these cases and Figure \ref{figura_3.4} for the bifurcation diagram in the variables $a_2$ and $b_2$. Define
$$
p_2^-=\left(\dfrac{2a_1\sqrt{\mu}}{\sqrt{a_1b_1}},\mu\right) \quad \textrm{and} \quad p_2^+=\left(-\dfrac{2a_1\sqrt{\mu}}{\sqrt{a_1b_1}},\mu\right).
$$

The orbit of \(X^+\) through $p_0=(0,-\mu)$ intersects the section
$L^+$ at two points, namely \(p_2^-\) and \(p_2^+\).

We denote by \(R_2^-\) (respectively \(R_2^+\))
the connected component of the orbit of \(X^+\) joining
$p_0$ to \(p_2^-\) (respectively \(p_2^+\)). See Figure \ref{figura_11.1}.

\begin{figure}[ht]
	\begin{center}
		\begin{overpic}[width=8cm]{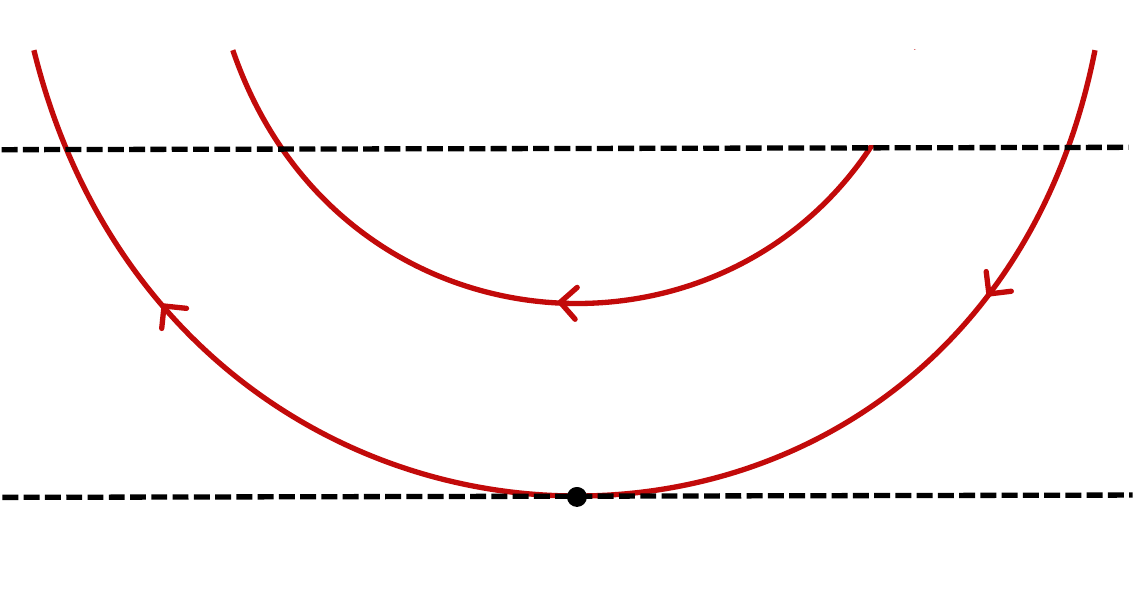}
            \put(5,45){$R_2^-$}
            \put(88,45){$R_2^+$}
            \put(46,5){$p_0$}
            \put(0,34){$p_2^-$}
            \put(95,34){$p_2^+$}
		   \end{overpic}
	\end{center}
\caption{Dynamics between points $p_2^-$ and $p_2^+$.}\label{figura_11.1}
\end{figure}
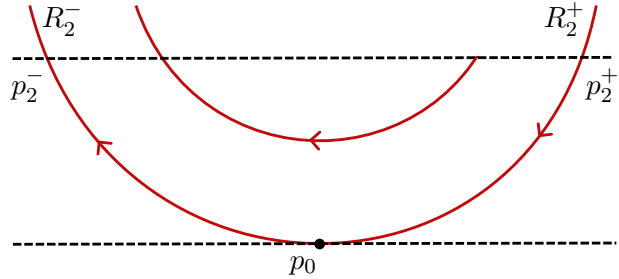

The curve \(R_2^+\) separates \(\Sigma^+\) into two connected components. By the \emph{right-hand side of \(R_2^+\)} we mean the right connected component together with the curve \(R_2^+\) itself. Analogously, the \emph{left-hand side of \(R_2^+\)} denotes the left connected component together with the curve \(R_2^+\).

\vspace{0.5cm}

$\bullet$ \textbf{Case 1.3.9: $a_1<0, b_1<0, a_2 > 0, b_2 = 0$.} In this case, taking an initial condition $(x_0,y_0)$ with $y_0>\mu$ on the right-hand side of \(R_2^+\), the trajectory has $\omega$-limit set $\{(+\infty,-\mu)\}$. Taking an initial condition $(x_0,y_0)$ with $y_0>\mu$ on the left-hand side of \(R_2^+\), the trajectory has $\omega$-limit set $\{(-\infty,+\infty)\}$. Finally, taking an initial condition $(x_0,y_0)$ with $y_0\leq\mu$, the trajectory has $\omega$-limit set $\{(+\infty,y_0)\}$.

$\bullet$ \textbf{Case 1.3.10: $a_1<0, b_1<0, a_2 > 0, b_2>0$.} To analyze the dynamics of the trajectories, we will use Poincaré first return map. See section \ref{subsec2}.

$\bullet$ \textbf{Case 1.3.11: $a_1<0, b_1<0, a_2 = 0, b_2>0$.} To analyze the dynamics of the trajectories, we will use Poincaré first return map. See section \ref{subsec2}.

$\bullet$ \textbf{Case 1.3.12: $a_1<0, b_1<0, a_2 < 0, b_2>0$.} To analyze the dynamics of the trajectories, we will use Poincaré first return map. See section \ref{subsec2}.

$\bullet$ \textbf{Case 1.3.13: $a_1<0, b_1<0, a_2<0, b_2=0$.} In this case, taking an initial condition $(x_0,y_0)$ with $y_0>\mu$ on the right-hand side of \(R_2^+\), the trajectory has $\omega$-limit set $\{(-\infty,-\mu)\}$. Taking an initial condition $(x_0,y_0)$ with $y_0>\mu$ on the left-hand side of \(R_2^+\), the trajectory has $\omega$-limit set $\{(-\infty,+\infty)\}$. Finally, taking an initial condition $(x_0,y_0)$ with $y_0\leq\mu$, the trajectory has $\omega$-limit set $\{(-\infty,y_0)\}$.

$\bullet$ \textbf{Case 1.3.14: $a_1<0, b_1<0, a_2<0, b_2<0$.} In this case, taking an initial condition $(x_0,y_0)$ with $y_0>\mu$ on the right-hand side of \(R_2^+\), the trajectory has $\omega$-limit set $\{(-\infty,-\infty)\}$. Taking an initial condition $(x_0,y_0)$ with $y_0>\mu$ on the left-hand side of \(R_2^+\), the trajectory has $\omega$-limit set $\{(-\infty,+\infty)\}$. Finally, taking an initial condition $(x_0,y_0)$ with $y_0\leq\mu$, the trajectory has $\omega$-limit set $\{(-\infty,-\infty)\}$.

$\bullet$ \textbf{Case 1.3.15: $a_1<0, b_1<0, a_2=0, b_2<0$.} In this case, taking an initial condition $(x_0,y_0)$ with $y_0>\mu$ on the right-hand side of \(R_2^+\), the trajectory has $\omega$-limit set
\[
\left\{ \left(\sqrt{x_2+\dfrac{2a_1(\mu-y_0)}{b_1}},-\infty\right) \right\} .
\]
Taking an initial condition $(x_0,y_0)$ with $y_0>\mu$ on the left-hand side of \(R_2^+\), the trajectory has $\omega$-limit set $\{(-\infty,+\infty)\}$. Finally, taking an initial condition $(x_0,y_0)$ with $y_0\leq\mu$, the trajectory has $\omega$-limit set $\{(x_0,-\infty)\}$.

$\bullet$ \textbf{Case 1.3.16: $a_1<0, b_1<0, a_2>0, b_2<0$.} In this case, taking an initial condition $(x_0,y_0)$ with $y_0>\mu$ on the right-hand side of \(R_2^+\), the trajectory has $\omega$-limit set $\{(+\infty,-\infty)\}$. Taking an initial condition $(x_0,y_0)$ with $y_0>\mu$ on the left-hand side of \(R_2^+\), the trajectory has $\omega$-limit set $\{(-\infty,+\infty)\}$. Finally, taking an initial condition $(x_0,y_0)$ with $y_0\leq\mu$, the trajectory has $\omega$-limit set $\{(+\infty,-\infty)\}$.

\begin{figure}[ht]
	\begin{center}
		\begin{overpic}[width=7cm]{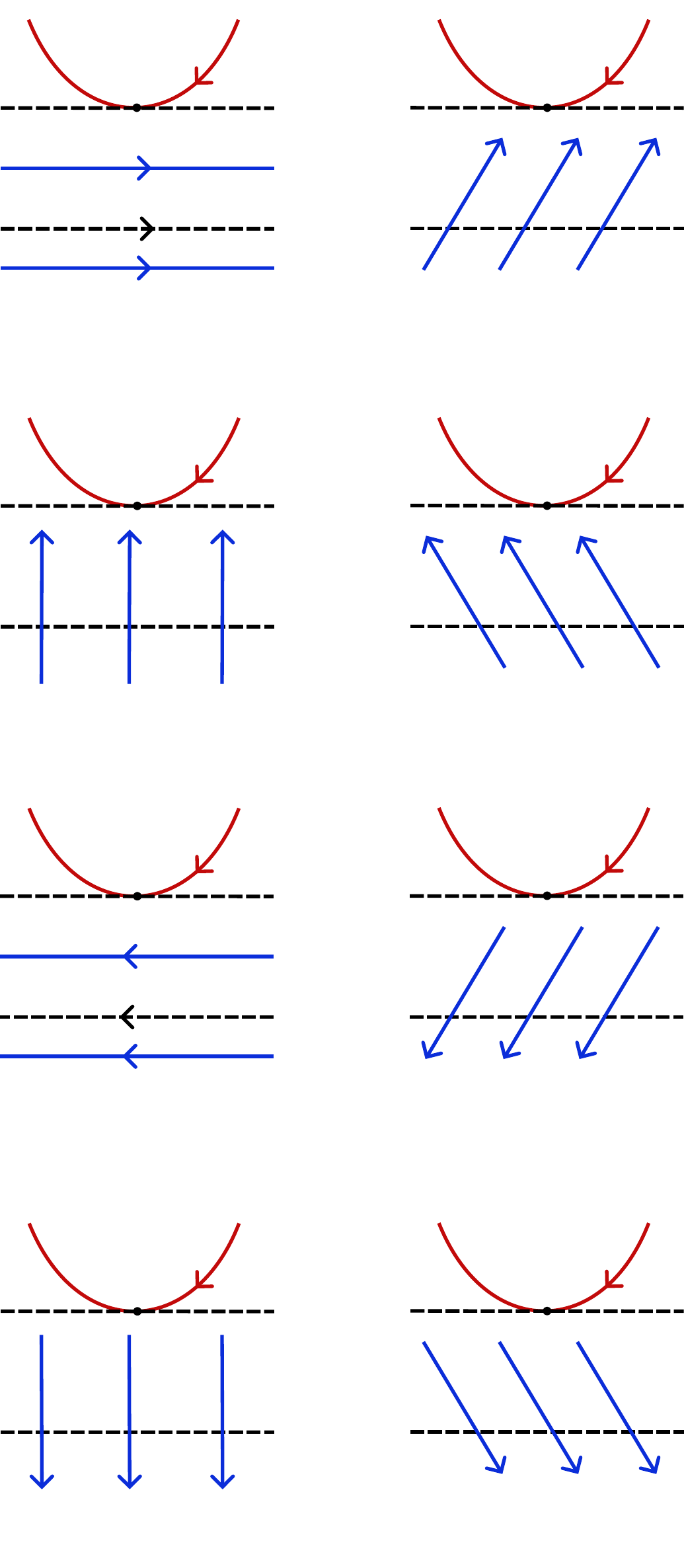}
			\put(4,78){Case $1.3.9$}
            \put(29,78){Case $1.3.10$}
            \put(4,53){Case $1.3.11$}
            \put(29,53){Case $1.3.12$}
            \put(4,28){Case $1.3.13$}
            \put(29,28){Case $1.3.14$}
            \put(4,1){Case $1.3.15$}
            \put(29,1){Case $1.3.16$}
		\end{overpic}
	\end{center}
	\caption{Dynamics for condition $a_1<0$ and $b_1<0$.}\label{figura_10.1}
\end{figure}
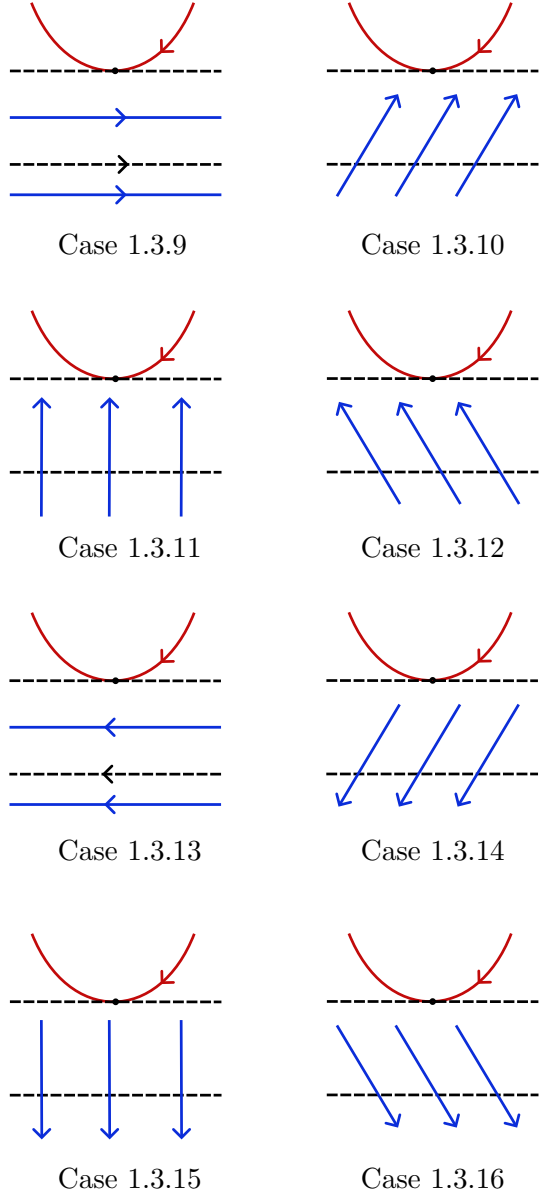

\begin{figure}[ht]
	\begin{center}
		\begin{overpic}[width=9cm]{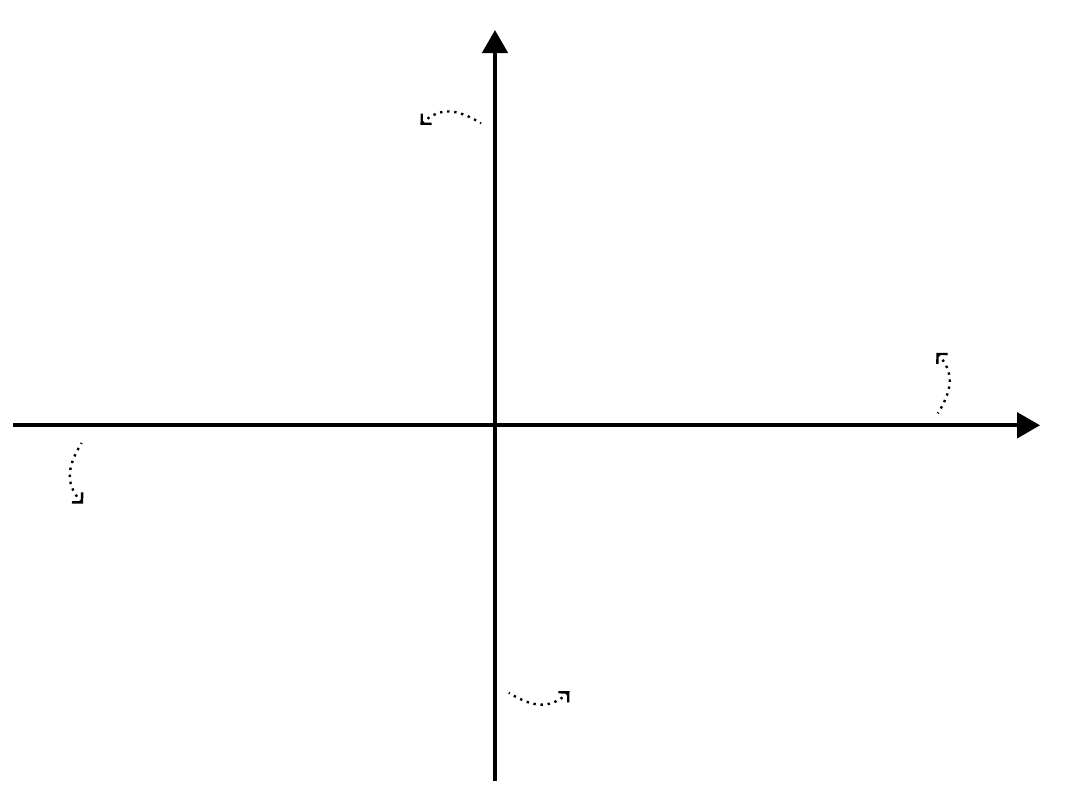}
            \put(97,30){$a_2$}
            \put(48,70){$b_2$}
			\put(75,45){Case $1.3.9$}
            \put(55,55){Case $1.3.10$}
            \put(23,55){Case $1.3.11$}
            \put(5,45){Case $1.3.12$}
            \put(5,20){Case $1.3.13$}
            \put(23,10){Case $1.3.14$}
            \put(55,10){Case $1.3.15$}
            \put(75,20){Case $1.3.16$}
		\end{overpic}
	\end{center}
	\caption{Bifurcation diagram in the variables $a_2$ and $b_2$, for  $a_1<0$, $b_1<0$.}\label{figura_3.4}
\end{figure}

\subsubsection{\textbf{Poincaré first return map}}\label{subsec2}

In this section, via the Poincaré first return map, let us describe the dynamics of the trajectories in Cases 1.2.2, 1.2.3, 1.2.4, 1.2.10, 1.2.11, 1.2.12, 1.3.2, 1.3.3, 1.3.4, 1.3.10, 1.3.11 and 1.3.12. Note that for all these cases $b_2\neq0$.

Assume that $a_2\neq0$ and consider $L^+$ given in \eqref{eqL}.

The Poincaré first return map is
\[
P:D(P)\subset L^+\longrightarrow L^+,
\]
where
\[
D(P)=\{p\in L^+:\text{the positive orbit of }p
\text{ returns to }L^+\}.
\]
Suppose an initial condition $p_1=\left(x_1, \mu\right)\in D(P)$ under the action of the vector field $X^+$. The solutions for this vector field are given by
$$
x^+\left(t\right)=a_1t+x_1 \quad \textrm{and} \quad y^+\left(t\right)=\dfrac{1}{2}(a_1b_1t^2+2b_1tx_1+2\mu). 
$$
The time which the flow takes to reach $L^-$ is 
\begin{equation}\label{eq2}
	t=-\dfrac{b_1x_1+\sqrt{b_1(b_1x_1^2-4a_1\mu)}}{a_1b_1}.
\end{equation}
Substituting \eqref{eq2} in the component solution $x^+$,
\[
x^+\left(t\right):=x_2=-\dfrac{\sqrt{b_1(b_1x_1^2-4a_1\mu)}}{b_1}.
\]
From this moment, the flow is governed by the vector field $X^-$ located under the hysteresis region and the solution for this vector field from the initial position $p_2=(x_2,-\mu)$ is given by 
\[
	x^-(t)=a_2t-\dfrac{\sqrt{b_1(b_1x_1^2-4a_1\mu)}}{b_1} \quad \textrm{and} \quad y^-\left(t\right)=b_2t-\mu.
\]
The time that this flow takes to achieve $L^+$ is $t=2\mu/b_2$ and substituting this time $t$ in $x^-$ gives
\[
x^-\left(\dfrac{2\mu}{b_2}\right)=\dfrac{2a_2\mu}{b_2}-\dfrac{\sqrt{b_1(b_1x_1^2-4a_1\mu)}}{b_1}.
\]
Hence, the Poincaré first return map is 
\begin{equation}\label{eqP}
    P\left(x\right)=\dfrac{2a_2\mu}{b_2}-\dfrac{\sqrt{b_1(b_1x^2-4a_1\mu)}}{b_1}
\end{equation}
and the derivative of $P$ is	
\begin{equation}\label{eqP'}
    P'\left(x\right)=-\dfrac{b_1x}{\sqrt{b_1(b_1x^2-4a_1\mu)}}.
\end{equation}

To find the fixed points of the Poincaré first return map $P$ and their stability, consider the following lemmas.

\begin{lemma}\label{lema1}
Let
\[
P(x)=A-B\sqrt{R(x)},
\]
where
\begin{equation}\label{eqABR}
    A=\frac{2a_2\mu}{b_2},
\qquad
B=\frac{1}{b_1},
\qquad
R(x)=b_1\left(b_1x^2-4a_1\mu\right).
\end{equation}

A point \(x^*\in\mathbb{R}\) is a fixed point of $P$, that is,
\[
P(x^*)=x^*,
\]
if and only if the following conditions hold:
\begin{equation}\label{eq1lema}
    (A-x^*)^2=B^2R(x^*),
\end{equation}
\begin{equation}\label{eq2lema}
    R(x^*)\geq 0,
\end{equation}
and
\begin{equation}\label{eq3lema}
    B(A-x^*)\ge 0.
\end{equation}
Futhermore, 
\[
    x^*=\dfrac{a_1 b_2}{a_2 b_1} + \dfrac{a_2 \mu}{b_2}.
\]

\end{lemma}

\begin{proof}
The fixed point equation
\[
x^*=A-B\sqrt{R(x^*)}
\]
is equivalent to
\begin{equation}\label{eq4lema}
    A-x^*=B\sqrt{R(x^*)}.
\end{equation}

Since \(\sqrt{R(x^*)}\) is defined only when \(R(x^*)\ge0\), every fixed point must satisfy
\[
R(x^*)\ge0.
\]
Moreover, because \(\sqrt{R(x^*)}\ge0\), the right-hand side of equation \eqref{eq4lema} has the same sign as \(B\), and therefore
\[
B(A-x^*)\ge0.
\]
Finally, squaring both sides of equation \eqref{eq4lema}, yields
\[
(A-x^*)^2=B^2R(x^*).
\]

Conversely, suppose that
\[
(A-x)^2=B^2R(x),
\qquad
R(x)\ge0,
\qquad
B(A-x)\ge0.
\]
Since \(R(x)\ge0\), taking square roots gives
\[
|A-x|=|B|\sqrt{R(x)}.
\]
The sign condition \(B(A-x)\ge0\) implies that \(A-x\) and \(B\) have the same sign, hence
\[
A-x=B\sqrt{R(x)}.
\]
Therefore,
\[
x=A-B\sqrt{R(x)}=P(x),
\]
showing that \(x\) is a fixed point of $P$.

Futhermore, 
\[
    x^*=A-B\sqrt{R(x^*)}=\dfrac{a_1 b_2}{a_2 b_1} + \dfrac{a_2 \mu}{b_2}.
\]
\end{proof}

Throughout this section we restrict our attention to the generic case in
which the fixed point belongs to the interior of the domain of the first
return map. More precisely, we assume that
\[
R(x^*)>0\Leftrightarrow a_2^2b_1\mu\neq a_1b_2^2,
\]
or equivalently,
\[
x^*\notin\partial D(P),
\]
where
\[
\partial D(P)
=
\{x\in D(P):R(x)=0\}.
\]
This assumption guarantees that the first return map is continuously
differentiable in a neighborhood of the fixed point. Indeed, by \eqref{eqP'} and \eqref{eqABR},

\[
P'(x)
=
-
\frac{b_1x}
{\sqrt{R(x)}},
\]
so the derivative is not defined whenever \(R(x)=0\). Consequently, the
usual stability criterion based on the derivative of $P$ at the fixed point can not be applied when the fixed point lies on the boundary of the domain.

\begin{remark}
The excluded case
\[
R(x^*)=0
\]
corresponds to a degenerate configuration in which the periodic orbit is
created or destroyed exactly at the boundary of the domain of the first
return map. Since the map fails to be differentiable at this point, its
stability requires a separate analysis, which is beyond the scope of the
present work.
\end{remark}

\begin{lemma}\label{lema2}
Let $P$ be the Poincaré first return map given by \eqref{eqP}. Then

\[
a_1>0,\ a_2>0
\quad\Longrightarrow\quad
P(x)>x,
\]
for every $x\in D(P)$, while

\[
a_1<0,\ a_2<0
\quad\Longrightarrow\quad
P(x)<x,
\]
for every $x\in D(P)$.

Consequently, whenever $a_1$ and $a_2$ have the same sign, the map $P$ admits no fixed points.
\end{lemma}

\begin{proof}
Consider the initial condition $(x,\mu) \in L^+$.

Suppose first that $a_1>0$ and $a_2>0$. The trajectory evolves under the vector field $X^+$ and
\[
\dfrac{dx^+}{dt}=a_1>0,
\]
i.e., the $x$-coordinate is strictly increasing until the trajectory reaches the section $L^-$. After switching, the trajectory evolves under
the vector field $X^-$. Since
\[
\dfrac{dx^-}{dt}=a_2>0,
\]
the $x$-coordinate remains strictly increasing until the trajectory returns to the section $L^+$. Consequently, the abscissa of the return point is strictly larger than the abscissa of the initial point, that is,
$
P(x)>x.
$

Now suppose that $a_1<0$ and $a_2<0$. How the flow is governed by the vector field $X^+$ and
\[
\dfrac{dx^+}{dt}=a_1<0,
\]
so the $x$-coordinate is strictly decreasing until the first switching.
During the return,
\[
\dfrac{dx^-}{dt}=a_2<0,
\]
i.e., the $x$-coordinate continues to decrease until the trajectory returns to the
section $L^+$. Hence the return point satisfies
$
P(x)<x.
$
Therefore,
\[
(P(x)-x)a_1>0,
\qquad
\forall\,x\in D(P),
\]
which implies that the equation
$
P(x)=x
$
has no solution. Hence, the Poincaré first return map admits no fixed
points.
\end{proof}

\begin{lemma}\label{lema3}
Assume that the Poincaré first return map $P$ given by \eqref{eqP} admits the fixed point $x^*$, i.e.,
\[
x^*=
\frac{a_1b_2}{a_2b_1}
+
\frac{a_2\mu}{b_2}.
\]

Then

\[
|P'(x^*)|<1
\quad\Longleftrightarrow\quad
a_1b_1<0,
\]

and

\[
|P'(x^*)|>1
\quad\Longleftrightarrow\quad
a_1b_1>0.
\]

Consequently, whenever the periodic orbit exists, it is attracting if
\[a_1b_1<0\] and repelling if \[a_1b_1>0.\]
\end{lemma}

\begin{proof}
From Equation \eqref{eqP'},
\[
|P'(x^*)|
=
\left|
\frac{
a_1b_2^2+a_2^2b_1\mu
}{
a_2^2b_1\mu-a_1b_2^2
}
\right|.
\]

The condition \(|P'(x^*)|<1\) is equivalent to
\[
\left|
a_1b_2^2+a_2^2b_1\mu
\right|
<
\left|
a_2^2b_1\mu-a_1b_2^2
\right|.
\]

Squaring both sides and simplifying gives
\[
(a_1b_2^2)(a_2^2b_1\mu)<0.
\]

Since
\[
a_2^2>0,
\qquad
b_2^2>0,
\qquad
\mu>0,
\]
the latter condition is equivalent to
\[
a_1b_1<0.
\]

Hence
\[
|P'(x^*)|<1
\quad\Longleftrightarrow\quad
a_1b_1<0.
\]

Similarly,
\[
|P'(x^*)|>1
\quad\Longleftrightarrow\quad
a_1b_1>0.
\]

Therefore the periodic orbit is attracting when \(a_1b_1<0\) and repelling when \(a_1b_1>0\).
\end{proof}

The previous lemma shows that the stability of the periodic orbit can be
expressed entirely in terms of the sign of the product $a_1b_1$.
Since the sign of $a_1b_1$ determines the visibility of the fold point of
$X^+$, the stability admits a simple geometric interpretation.

\begin{corollary}
Assume that the Poincaré first return map $P$ admits a fixed point. Then the corresponding periodic orbit is attracting if and only if the fold
point of $X^+$ is invisible, and repelling if and only if the fold point of $X^+$ is visible.
\end{corollary}

Now, let us study the dynamics of the missing cases.

For Cases 1.2.3, 1.2.11, 1.3.3 and 1.3.11, we have $a_2=0$. Note that the previous expression for the fixed point is no longer valid. In this case, the first return map reduces to

\[
P(x)=
-\frac{\sqrt{b_1(b_1x^2-4a_1\mu)}}{b_1}.
\]

Assume that $x^*$ is a fixed point. Then

\[
x^*=
-\frac{\sqrt{b_1(b_1(x^*)^2-4a_1\mu)}}{b_1}.
\]

Squaring both sides yields

\[
b_1^2(x^*)^2
=
b_1\left(b_1(x^*)^2-4a_1\mu\right),
\]

which simplifies to

\[
4a_1b_1\mu=0.
\]

Since $a_1\neq0$, $b_1\neq0$ and $\mu>0$, we obtain a contradiction.
Therefore, the first return map admits no fixed points when $a_2=0$. So, for Cases 1.2.3 and 1.2.11, taking an initial condition $(x_0,y_0)$, the trajectory has a monotone zig-zag behavior restricted to HB region, with the $x$-coordinate going to $+ \infty$ for the Case 1.2.3 and the $x$-coordinate going to $-\infty$ for the Case 1.2.11.

For Cases 1.3.3 and 1.3.11 define
\[
n_3:=\dfrac{\left|x_0+\dfrac{2a_1\sqrt{\mu}}{\sqrt{a_1b_1}}\right|}{\left|
-2x_0
-\dfrac{\sqrt{b_1\left(b_1x_2+2a_1(\mu-y_0)\right)}}{b_1}
-\dfrac{\sqrt{b_1\left(b_1x_2-2a_1(y_0+\mu)\right)}}{b_1}
\right|}
\]
and
\[
n_4:=\dfrac{\left|x_0+\dfrac{2a_1\sqrt{\mu}}{\sqrt{a_1b_1}}\right|}{\left|
-x_0
-\dfrac{\sqrt{b_1\left(b_1x_2-2a_1(y_0+\mu)\right)}}{b_1}
\right|}.
\]
For Case 1.3.3, taking an initial condition $(x_0,y_0)$ with $y_0>\mu$ on the left-hand side of \(R_1^-\), the trajectory has a monotone zig-zag behavior restricted to HB region, with
$$
N_3:=\left\lceil n_3\right\rceil+1\geq2
$$
intersections with the line $y=\mu$ and has $\omega$-limit set $\left\{(+\infty,+\infty\right)\}$. Taking an initial condition $(x_0,y_0)$, with $y_0\leq\mu$ and
\[
x_0\leq-\dfrac{2a_1\sqrt{\mu}}{\sqrt{a_1b_1}},
\]
the trajectory has a monotone zig-zag behavior restricted to HB region, with
$$
N_4:=\left\lceil n_4\right\rceil+1\geq2
$$
intersections with the line $y=\mu$ and has $\omega$-limit set $\left\{(+\infty,+\infty\right)\}$.

Now, taking an initial condition $(x_0,y_0)$ with $y_0>\mu$ on the right-hand side of \(R_1^-\) or an initial condition $(x_0,y_0)$ with $y_0\leq\mu$ and
\[
x_0>-\dfrac{2a_1\sqrt{\mu}}{\sqrt{a_1b_1}},
\]
the trajectory has $\omega$-limit set $\{(+\infty,+\infty)\}$.

For Case 1.3.11, taking an initial condition $(x_0,y_0)$ with $y_0>\mu$ on the right-hand side of \(R_2^+\), the trajectory has a monotone zig-zag behavior restricted to HB region, with $N_3$ intersections with the line $y=\mu$ and has $\omega$-limit set $\left\{(-\infty,+\infty\right)\}$. Taking an initial condition $(x_0,y_0)$, with $y_0\leq\mu$ and
\[
x_0\geq-\dfrac{2a_1\sqrt{\mu}}{\sqrt{a_1b_1}},
\]
the trajectory has a monotone zig-zag behavior restricted to HB region, with $N_4$ intersections with the line $y=\mu$ and has $\omega$-limit set $\left\{(-\infty,+\infty\right)\}$.

Now, taking an initial condition $(x_0,y_0)$ with $y_0>\mu$ on the left-hand side of \(R_2^+\) or an initial condition $(x_0,y_0)$ with $y_0\leq\mu$ and
\[
x_0<-\dfrac{2a_1\sqrt{\mu}}{\sqrt{a_1b_1}},
\]
the trajectory has $\omega$-limit set $\{(-\infty,+\infty)\}$.

For Cases 1.2.2, 1.2.12, 1.3.2 and 1.3.12, the Poincaré map $P$ does not have a fixed point, by Lemma \ref{lema2}. Therefore, for Cases 1.2.2 and 1.2.12, taking an initial condition $(x_0,y_0)$, the trajectory has a monotone zig-zag behavior restricted to HB region, with the $x$-coordinate going to $+ \infty$ for the Case 1.2.2 and the $x$-coordinate going to $-\infty$ for the Case 1.2.12.

For Case 1.3.2, consider the half-line 
$$
R_3=\left\{(x,y)\in\mathbb{R}^2; \,y=\dfrac{b_2}{a_2}\left(x+\dfrac{2a_1\sqrt{\mu}}{\sqrt{a_1b_1}}\right)+\mu, \, x<0\right\},
$$
$$
n_5:=\dfrac{\left|x_0+\dfrac{2a_1\sqrt{\mu}}{\sqrt{a_1b_1}}\right|}{\left|
2\left(
-x_0
+\dfrac{a_2\mu}{b_2}
-\dfrac{\sqrt{b_1\left(b_1x_2-2a_1y_0+2a_1\mu\right)}}{b_1}
\right)
\right|}
$$
and
$$
n_6:=\dfrac{\left|x_0+\dfrac{2a_1\sqrt{\mu}}{\sqrt{a_1b_1}}\right|}{\left|
-x_0
-\dfrac{a_2(y_0-3\mu)}{b_2}
-\dfrac{\sqrt{b_1\left(b_1x_1^2-4a_1\mu\right)}}{b_1}
\right|}
$$
Taking an initial condition $(x_0,y_0)$ with $y_0>\mu$ on the left-hand side of \(R_1^-\), the trajectory has a monotone zig-zag behavior restricted to HB region, with
$$
N_5:=\left\lceil n_5\right\rceil+1\geq2
$$
intersections with the line $y=\mu$ and has $\omega$-limit set $(+\infty, +\infty)$. Taking an initial condition $(x_0,y_0)$ with $y_0\leq\mu$ lying on or above half-line $R_3$, the trajectory has a monotone zig-zag behavior restricted to HB region, with
$$
N_6:=\left\lceil n_6\right\rceil+1\geq2
$$
intersections with the line $y=\mu$ and has $\omega$-limit set $(+\infty, +\infty)$.

Now, taking an initial condition $(x_0,y_0)$ below the half-line $R_3$ on the right-hand side of \(R_1^-\), the trajectory  will converge to $(+\infty, +\infty)$.

For Case 1.3.12, consider the half-line 
$$
R_4=\left\{(x,y)\in\mathbb{R}^2; \,y=\dfrac{b_2}{a_2}\left(-x+\dfrac{2a_1\sqrt{\mu}}{\sqrt{a_1b_1}}\right)+\mu, \, x>0\right\}.
$$
Taking an initial condition $(x_0,y_0)$ with $y_0>\mu$ on the right-hand side of $R_2^+$, the trajectory has a monotone zig-zag behavior restricted to HB region, with $N_5$ intersections with the line $y=\mu$ and has $\omega$-limit set $ \{ (-\infty, +\infty) \}$. Taking an initial condition $(x_0,y_0)$ with $y_0\leq\mu$ lying on or above half-line $R_4$, the trajectory has a monotone zig-zag behavior restricted to HB region, with $N_6$ intersections with the line $y=\mu$ and has $\omega$-limit set $\{  (-\infty, +\infty) \}$.

Now, taking an initial condition $(x_0,y_0)$ below the half-line $R_4$ on the left-hand side of \(R_2^+\), the trajectory  will converge to $\{  (-\infty, +\infty) \}$.

For Case 1.2.4, we have
\begin{equation}\label{eq1.2.4}
    a_1>0, b_1<0, a_2 < 0 \quad \textrm{and} \quad b_2>0,
\end{equation}
for Case 1.2.10,
\begin{equation}\label{eq1.2.10}
    a_1<0, b_1>0, a_2 > 0 \quad \textrm{and} \quad b_2>0,
\end{equation}
for Case 1.3.4,
\begin{equation}\label{eq1.3.4}
    a_1>0, b_1>0, a_2 < 0 \quad \textrm{and} \quad b_2>0,
\end{equation}
and for Case 1.3.10,
\begin{equation}\label{eq1.3.10}
    a_1<0, b_1<0, a_2 > 0 \quad \textrm{and} \quad b_2>0.
\end{equation}
By Lemma \ref{lema1}, the only possible fixed point is of the form
\[
x^*=
\frac{a_1b_2}{a_2b_1}+\frac{a_2\mu}{b_2}.
\]
Substituting $x^*$ into the expressions
\[
(A-x)^2=B^2R(x),
\quad
R(x),
\quad
B(A-x),
\]
and taking into account \eqref{eq1.2.4}, \eqref{eq1.2.10}, \eqref{eq1.3.4} and \eqref{eq1.3.10}, the conditions \eqref{eq1lema}, \eqref{eq2lema} and \eqref{eq3lema} of Lemma \ref{lema1} are satisfied.

Therefore, by Lemma \ref{lema3}, we conclude that the periodic orbit corresponding to $x^*$ is globally attractive for Cases 1.2.4 and 1.2.10 and repulsive for Cases 1.3.4 and 1.3.10.

For Case 1.3.4, taking an initial condition $(x_0,y_0)$, with $x_0<x^*$ and outside the trajectories of the $X^+$ and $X^-$ that pass through point $(x^*,\mu)$, the trajectory has a monotone zig-zag behavior restricted to HB region, with the $x$-coordinate going to $- \infty$. Taking an initial condition $(x_0,y_0)$, with $x_0>x^*$ and outside the trajectories of the $X^+$ and $X^-$ that pass through point $(x^*,\mu)$, the trajectory has $\omega$-limit set $\{ (+\infty, +\infty) \}$.

For Case 1.3.10, taking an initial condition $(x_0,y_0)$, with $x_0>x^*$ and outside the trajectories of the $X^+$ and $X^-$ that pass through point $(x^*,\mu)$, the trajectory has a monotone zig-zag behavior restricted to HB region, with the $x$-coordinate going to $+\infty$. Taking an initial condition $(x_0,y_0)$, with $x_0<x^*$ and outside the trajectories of the $X^+$ and $X^-$ that pass through point $(x^*,\mu)$, the trajectory has $\omega$-limit set $\{ (-\infty, +\infty) \}$.

\begin{lemma}\label{lema4}
Assume that the Poincaré first return map $P$ admits the fixed point
\[
x^*=
\frac{a_1b_2}{a_2b_1}
+
\frac{a_2\mu}{b_2}.
\]
Then $P$ has no periodic point of minimal period two.
\end{lemma}

\begin{proof}
A fixed point of $P^2$ is either a fixed point of $P$ or a periodic point
of minimal period two. Thus, it suffices to prove that every solution of
$
P^2(x)=x
$
also satisfies
$
P(x)=x.
$

Since
\[
P(x)=A-B\sqrt{R(x)},
\]
we have
\[
P^2(x)=A-B\sqrt{R(P(x))}.
\]
The equation \(P^2(x)=x\) is therefore equivalent to
\[
A-x=B\sqrt{R(P(x))}.
\]
Applying Lemma \ref{lema1} to the map $P^2$, and keeping the sign
conditions required by the square roots, the equation reduces to
\[
x=
\frac{a_1b_2}{a_2b_1}
+
\frac{a_2\mu}{b_2}
=x^*.
\]
Thus $P^2(x)=x$ has no solution other than $x^*$, i.e., the map $P$ has no periodic point of
minimal period two.
\end{proof}

%
%


\section{Conclusion}\label{conc}

In this paper we obtained a complete classification of the $\omega$-limit sets for a family of planar switched control systems with large hysteresis generated by two affine vector fields. Owing to the explicit form of the solutions, all possible parameter configurations were analyzed, allowing a complete description of the global asymptotic behavior.

The use of the Poincaré first return map made it possible to characterize the existence, uniqueness and stability of periodic orbits whenever such trajectories arise. In particular, explicit parameter conditions were obtained for the existence of attracting and repelling limit cycles, while the remaining parameter regions were shown to exhibit either monotone zig-zag dynamics inside the hysteresis band or unbounded trajectories.

Beyond the classification of limit sets, the obtained bifurcation diagrams reveal the global organization of the dynamics. As the parameters vary, the system undergoes transitions corresponding to the creation or disappearance of periodic orbits, exchanges of stability, degenerate configurations giving rise to continuum of periodic trajectories, and changes between bounded and unbounded asymptotic behavior.

Therefore, the present work provides both a complete classification of the limit sets and a global qualitative description of the bifurcation structure associated with this class of hysteretic control systems.

Future research includes extending these results to nonlinear vector fields, higher-dimensional hysteretic systems, and Filippov systems with multiple switching manifolds.

\section*{Acknowledgments}

Tiago Carvalho is partially supported by S\~{a}o Paulo Research Foundation (FAPESP grants \# 2024/15612-6, \#2021/12395-6, \#2019/10269-3 and \#2022/02819-6) and by Conselho Nacional de Desenvolvimento Cient\'{i}fico e Tecnol\'{o}gico (CNPq Grants 309378/2023-0 and 401974/2025-1).

Bruno de Souza Rangel is supported by S\~{a}o Paulo Research Foundation (FAPESP grant \#2023/18081-9).

\section*{Author Contributions}
All authors contributed equally to this work, including conception, design, analysis, and writing. All authors have read and approved the final manuscript.

\section*{Data availability}
Data sharing is not applicable to this article, as no data sets were generated or analyzed during the current study.

\section*{Declaration of interest}
The authors declare that they have no conflict of interest.

\end{document}